\documentclass[]{interact}

\usepackage{epstopdf}
\usepackage[caption=false]{subfig}

\usepackage[numbers,sort&compress]{natbib}
\bibpunct[, ]{[}{]}{,}{n}{,}{,}
\makeatletter
\def\NAT@def@citea{\def\@citea{\NAT@separator}}
\makeatother
\usepackage{hyperref}
\hypersetup{
	colorlinks=true,
	linkcolor=blue, 
	citecolor=red, 
	urlcolor=blue  } 
\usepackage{mathptmx}   

\theoremstyle{plain}
\newtheorem{theorem}{Theorem}[section]
\newtheorem{lemma}[theorem]{Lemma}
\newtheorem{corollary}[theorem]{Corollary}
\newtheorem{proposition}[theorem]{Proposition}

\theoremstyle{definition}
\newtheorem{definition}[theorem]{Definition}

\theoremstyle{remark}
\newtheorem{remark}{Remark}

\theoremstyle{plain}
{\bf}{\it}

\theoremstyle{plain}
 {\bf}{\it}

\theoremstyle{plain}
 {\bf}{\it}

\theoremstyle{plain}
 {\bf}{\it}

\theoremstyle{plain}
 {\bf}{\it}

\theoremstyle{plain}
 {\bf}{\it}

\theoremstyle{plain}
 {\bf}{\it}

\theoremstyle{plain}
 {\bf}{\it}

\theoremstyle{plain}
 {\bf}{\it}

\theoremstyle{plain}
 {\bf}{\it}

\theoremstyle{plain}
 {\bf}{\it}

\theoremstyle{plain}
 {\bf}{\it}

\theoremstyle{plain}
 {\bf}{\it}

\theoremstyle{plain}
 {\bf}{\it}

\theoremstyle{plain}
 {\bf}{\it}

\theoremstyle{plain}
 {\bf}{\it}

\theoremstyle{plain}
 {\bf}{\it}

\usepackage{color}

\newcommand{\al}{\alpha}
\newcommand{\be}{\beta}
\newcommand{\ga}{\gamma}

\newcommand{\de}{\delta}
\newcommand{\eps}{\varepsilon}
\newcommand{\bx}{\bar x}
\newcommand{\by}{\bar y}

\newcommand{\iv}{^{-1} }

\newcommand {\R} {\mathbb R}
\newcommand {\N} {\mathbb N}

\newcommand {\B} {\mathbb B}

\newcommand {\gph} {{\textrm{gph}}\,}
\newcommand {\dom} {{\textrm{dom}}\,}
\newcommand {\epi} {{\textrm{epi}}\,}

\newcommand {\bd} {{\textrm{bd}}\,}

\newcommand {\Er} {{\textup{Er}}\,}

\newcommand {\sd} {\partial}

\newcommand{\folgt}{$ \Rightarrow\ $}

\def\nbh{neighbourhood}
\def\es{\emptyset}
\def\lsc{lower semicontinuous}

\def\SVM{set-valued mapping}
\def\EVP{Ekeland variational principle}
\def\Fr{Fr\'echet}

\newcommand{\red}[1]{\textcolor{red}{#1}}

\newcommand{\qdtx}[1]{\quad\mbox{#1}\quad}
\newcommand{\AND}{\quad\mbox{and}\quad}
\newcounter{mycount}
\def\cnta{\setcounter{mycount}{\value{enumi}}}
\def\cntb{\setcounter{enumi}{\value{mycount}}}

\newcommand{\AK}[1]{\todo[inline]{AK {#1}}}


\renewcommand {\theenumi} {\rm(\roman{enumi})}

\begin{document}
\title{Error bounds revisited}

\author{
\name{Nguyen Duy Cuong\textsuperscript{a,b} and Alexander Y. Kruger\textsuperscript{a}}
\thanks{CONTACT Alexander Y. Kruger. Email: a.kruger@federation.edu.au}
\affil{\textsuperscript{a} Centre for Informatics and Applied Optimization, School of Engineering, IT and Physical Sciences, Federation University, POB 663, Ballarat, Vic,
3350, Australia; \textsuperscript{b} Department of Mathematics, College of Natural Sciences, Can Tho University, Vietnam}
\vspace{5mm}
{Dedicated to the memory of Prof Alexander Rubinov, a teacher and friend}
}
\maketitle

\begin{abstract}
We propose a unifying general framework of quantitative primal and dual sufficient and necessary error bound conditions covering linear and nonlinear, local and global settings.
The function is not assumed to possess any particular structure apart from the standard assumptions of lower semicontinuity in the case of sufficient conditions and (in some cases) convexity in the case of necessary conditions.
We expose the roles of the assumptions involved in the error bound assertions, in particular, on the underlying space: general metric, normed, Banach or Asplund.
Employing special collections of
slope operators,
we introduce a succinct form of sufficient error bound conditions, which allows one to combine in a single statement several different assertions: nonlocal and local primal space conditions in complete metric {spaces}, and subdifferential conditions in Banach and Asplund spaces.
\end{abstract}


\begin{amscode}
49J52; 49J53; 49K40; 90C30; 90C46
\end{amscode}


\section{Introduction}

Necessary and especially sufficient conditions for error bounds of (extended) real-valued functions have been a subject of intense study for more than half a century due to their numerous applications in optimization and variational analysis,
particularly in convergence analysis of iterative algorithms, penalty functions, optimality conditions, weak sharp minima, stability and well-posedness of solutions, (sub)regularity and calmness of \SVM s, and subdifferential calculus; see, e.g., \cite{Pan97,YeYe97,AzeCor04,Kru15,Kru16,YaoZhe16,AzeCor17,Iof17, KruLopYanZhu19}.

A huge number of sufficient and necessary conditions for error bounds have been obtained in the linear
\cite{Jou00,NgZhe01,WuYe01,AzeCorLuc02,WuYe02,WuYe03, AzeCor04,
NgaThe04,FabHenKruOut10, FabHenKruOut12,MenYan12,Kru15,ChuJey16.1,ChuJey16.2, LiMenYan18, KruLopYanZhu19},
as well as more subtle nonlinear (mostly H\"older)
\cite{NgZhe00, WuYe02,
CorMot08,NgaThe08,NgaThe09,
AzeCor14,ChaChe14,Kru15.2,Kru16, YaoZhe16,AzeCor17,
LiMorNghPha18, KruLopYanZhu19,ZhaZhe19}
settings.

Given an extended-real-valued function $f: X \rightarrow\R_\infty:=\mathbb{R}\cup \{+\infty\}$ on a metric space $X$, and $\mu\in]0,+\infty]$, we
employ the relatively standard notations
$$[f\le0]:=\{x\in X\mid f(x)\le0\},\quad
[0<f<\mu]:=\{x\in X\mid 0<f(x)<\mu\}.$$
The sets $[f>0]$, $[f<\mu]$ and $[f\le\mu]$ are defined in a similar way.

\begin{definition}\label{D0}
Suppose $X$ is a metric space, $f:X \rightarrow \mathbb{R}_\infty$, and $\tau>0$.
The function $f$ admits a $\tau-$error bound at $\bx\in X$ if there exist $\delta\in]0,+\infty]$ and $\mu\in]0,+\infty]$ such that
\begin{align}\label{D0-1}
\tau d(x,[f \le 0])\le f(x)
\end{align}
for all $x\in B_\de(\bx)\cap[0<f<\mu]$, and either $\bx\in[f\le0]$ or $\de=+\infty$.
\end{definition}

The value of $\tau$ in Definition~\ref{D0} obviously depends on the values of $\delta$ and $\mu$.
We will often say that $f$ admits a $\tau-$error bound at $\bx$ with $\delta$ and $\mu$.

Definition~\ref{D0} combines the cases of local and global error bounds that are often treated separately.
The conventional local error bound property corresponds to the case $\bx\in[f\le0]$, and $\de$
being a (sufficiently small) finite number.
In this case, we say that $f$ admits a local $\tau-$error bound (with $\delta$ and $\mu$).
When $\delta=+\infty$, we have $B_\de(\bx)=X$, i.e. the error bound property in Definition~\ref{D0} is not related to any particular point, and we are in the setting of global error bounds.
In this case, we simply say that $f$ admits a global $\tau-$error bound.

The (local) error bound modulus of $f$ at $\bx\in[f\le 0]$ is defined as the exact upper bound of all $\tau>0$ such that $f$ admits a $\tau-$error bound at $\bx$ (with some $\delta\in]0,+\infty[$ and $\mu\in]0,+\infty]$), i.e.
\begin{align}\label{D0-2}
\Er f(\bx):=\liminf_{x\to\bx,\, f(x)>0}\frac{f(x)}{d(x,[f\le0])} =\liminf_{x\to\bx,\,f(x)\downarrow0}\frac{f(x)}{d(x,[f \le 0])}.
\end{align}

In the above definition and throughout the paper, we use the conventions $d(x,\es_X)=+\infty$, $\inf\es_\R=+\infty$ and $\frac{+\infty}{+\infty}=+\infty$.
(The last convention is only needed to accommodate for the trivial case $f\equiv+\infty$.)
The second equality in \eqref{D0-2} is straightforward.

By definition \eqref{D0-2}, $\Er f(\bx)\ge0$.
If $\Er f(\bx)=0$, then $f$ does not admit a $\tau-$error bound at $\bx$ for any $\tau>0$.

The next definition introduces error bounds in the nonlinear setting.
The nonlinearity is determined by a function $\varphi:\mathbb{R}_+ \rightarrow\mathbb{R}_+$ satisfying $\varphi(0)=0$ and $\varphi(t)>0$ if $t>0$.
The family of all such functions is denoted by $\mathcal{C}$.
We denote by $\mathcal{C}^1$ the subfamily of functions from $\mathcal{C}$ which satisfy
$\lim_{t\to+\infty}\varphi(t)=+\infty$, and
\if{
\todo{Should be waived.}
\NDC{31.8.20
In general, I think this condition can be dropped. But, looking at the proof of Theorem~\ref{T2.5}, this condition ensures the existence of
$\varphi\iv$.}
}\fi
are continuously differentiable on $]0,+\infty[$
with $\varphi'(t)>0$ for all $t>0$.
Obviously, if $\varphi\in\mathcal{C}^1$, then $\varphi\iv\in\mathcal{C}^1$.
Observe that,
for any $\alpha > 0$ and $q > 0$,
the function $\R_+\ni t\mapsto\alpha t^q$ belongs to $\mathcal{C}^1$.

\begin{definition}\label{D1}
Suppose $X$ is a metric space, $f:X \rightarrow \mathbb{R}_\infty$, and $\varphi\in\mathcal{C}$.
The function $f$ admits a $\varphi-$error bound at $\bx\in X$ if there exist $\delta\in]0,+\infty]$ and $\mu\in]0,+\infty]$ such that
\begin{align}\label{D1-1}
d(x,[f\le0])\le\varphi(f(x))
\end{align}
for all $x\in B_\de(\bx)\cap[0<f<\mu]$, and either $\bx\in[f\le0]$ or $\de=+\infty$.
\end{definition}

Under the conditions of Definition~\ref{D1}, we will often say that $f$ admits a $\varphi-$error bound at $\bx$ with $\delta$ and $\mu$.
When $\de<+\infty$, we say that $f$ admits a local $\varphi-$error bound at $\bx$ (with $\delta$ and $\mu$).
When $\de=+\infty$ (hence, $B_\de(\bx)=X$), we say that $f$ admits a global $\varphi-$error bound.

\begin{remark}\label{R1.1}
\begin{enumerate}
\item
The $\tau-$error bound property in Definition~\ref{D0} is a particular case of the $\varphi-$error bound property, corresponding to $\varphi$ being the linear function $t\mapsto\tau\iv t$.
\item
Any function $\varphi\in\mathcal{C}$ can be extended to the whole $\R$ by setting $\varphi(t)=0$ for all $t<0$.
Then one obviously has $[f\le 0]=[\varphi\circ f\le 0]$, and the seemingly more general $\varphi-$error bound property in Definition~\ref{D1} becomes the conventional $1-$error bound property for the composition function $\varphi\circ f$.
\item
The requirement $\lim_{t\to+\infty}\varphi(t)=+\infty$ in the definition of the family $\mathcal{C}^1$ is technical.
It is only needed to ensure that $\varphi\iv$ is defined on the whole $\R_+$.
Both conditions can be weakened.

In all assertions in the current paper involving \Fr\ subdifferentials it is sufficient to assume functions from $\mathcal{C}^1$ to be (not necessarily continuously) differentiable.
Continuous differentiability is only needed for assertions involving Clarke subdifferentials.
\item
It is not uncommon to consider nonlinear error bounds with inequality \eqref{D1-1} in Definition~\ref{D1} replaced with the following one:
\begin{align}\label{D1-2}
\psi(d(x,[f\le0]))\le f(x),
\end{align}
where $\psi\in\mathcal{C}$.
The models \eqref{D1-1} and \eqref{D1-2} are obviously equivalent with $\psi=\varphi\iv$ as long as one of the functions $\psi$ or $\varphi$ (hence, also the other one) is strictly increasing, which is the case, in particular, when $\varphi,\psi\in\mathcal{C}^1$.
\end{enumerate}
\end{remark}

Similar to the linear case, it makes sense to look for a smaller function $\varphi\in\mathcal{C}$ satisfying inequality \eqref{D1-1} (for the appropriate set of points $x\in X$).
It is not easy
to order functions on the whole set $\mathcal{C}$.
It is more practical to consider a subset of positive multiples of a given function $\varphi\in\mathcal{C}$.
Extending definition \eqref{D0-2}, we define the (local) $\varphi -$error bound modulus of $f$ at $\bx$ as
\begin{align}\label{D0-3}
\Er_\varphi f(\bx):=\liminf_{x\to\bx,\, f(x)>0}\frac{\varphi(f(x))}{d(x,[f\le0])}. 
\end{align}
It is easy to see that this is the reciprocal of the 
infimum
of all $\al>0$ such that $f$ admits an $(\al\varphi)-$error bound (with some $\de$ and $\mu$) at $\bx$.
When $\varphi(t)=t^q$ for some $q>0$ and all $t>0$, definition \eqref{D0-3} coincides with that of the modulus of $q-$order error bounds \cite{KruLopYanZhu19}.
In particular, with $q=1$ it reduces to \eqref{D0-2}.

In this paper, we do not aim to add some new sufficient or necessary conditions for error bounds to the large volume of existing ones (although some conditions in the subsequent sections are indeed new even in the linear setting).
Our main aim is to propose a comprehensive unifying general
(i.e. not assuming the function $f$ to have any particular structure) view on the theory of error bounds (linear and nonlinear, local and global), and clarify the relationships between the existing conditions including their hierarchy.
We expose the roles of the assumptions involved in the error bound assertions, in particular, on the underlying space: general metric, normed, Banach or Asplund.
Employing special collections of
slope operators defined below,
we introduce a succinct form of sufficient error bound conditions, which allows one to combine in a single statement several different assertions: nonlocal and local primal space conditions in complete metric space, and subdifferential conditions in Banach and Asplund spaces.

The hot topics of
error bounds for special families of functions and error bounds under uncertainty (see, e.g., \cite{ChuJey16.1,ChuJey16.2,LiMorNghPha18,LiMenYan18,DutMar}) are outside the scope of the current paper.

The core of the paper consists of three theorems treating linear (Theorem~\ref{T2.2}), nonlinear (Theorem~\ref{T2.4}) and `alternative' nonlinear (Theorem~\ref{T2.5}) error bound conditions that seem to cover all existing general
error bound results.
We show that the `nonlinear' Theorem~\ref{T2.4} is a straightforward consequence of the `linear' Theorem~\ref{T2.2},
while the `alternative nonlinear' Theorem~\ref{T2.5} is a straightforward consequence of the conventional Theorem~\ref{T2.4}.
In its turn, the original Theorem~\ref{T2.2} is a consequence of a preliminary statement -- Proposition~\ref{T2.1} -- treating the case when $x$ in \eqref{D0-1} is fixed, while encapsulating
all the main arguments used in the general statement.
Following Ioffe \cite[Basic lemma]{Iof00},
separate preliminary `fixed $x$' type statements
have been formulated by many authors; cf.
\cite{WuYe03, NgaThe04,NgaThe08,NgaThe09,Iof17,KruLopYanZhu19}.
Proposition~\ref{T2.1} seems to be the most comprehensive one.

\if{
\AK{29/08/20.
Sufficient conditions for global error bounds usually assume that $[f\le0]\ne\es$.
To be checked.}
\NDC{5.9.20
Papers of Ngai et al., Aze et al. mentioned this condition in the conclusion.
\cite{WuYe01} and \cite[Corollary 2.2]{HuaNg04} mentioned this condition in the assumption.
By the way, in \cite[Theorem 2.2]{WuYe03} they did not require $\bx\in[f\le 0]$ like what Aze et al. did.
}
}\fi

All the statements have the same structure, each combining several assertions that are mostly well known and are often formulated (and proved) as separate theorems:
\begin{enumerate}
\item
sufficient error bound conditions for a \lsc\ function on a complete metric space:
\begin{enumerate}
\item
nonlocal primal space conditions;
\item
infinitesimal primal space conditions in terms of slopes;
\item
in the setting of a Banach space, dual space conditions in terms of Clarke subdifferentials;
\item
in the setting of an Asplund space, dual space conditions in terms of \Fr\ subdifferentials;
\end{enumerate}
\item
nonlocal primal space necessary error bound conditions for a (not necessarily \lsc) function on a metric space (except  Theorem~\ref{T2.5});
\item
dual space necessary error bound conditions for a convex function on a normed space in terms of conventional convex subdifferentials.
\end{enumerate}

This paper seems to be the first attempt to
combine the above assertions in a single statement.
We believe that it not only makes the presentation shorter, but also clarifies the overall picture: it exposes the relationships between the assertions and the hierarchy of the sufficient conditions in part (i).
Most of the assertions in Theorem~\ref{T2.2} and to some extent also in Theorems~\ref{T2.4} and \ref{T2.5} have been proved multiple times by many authors, often with long
multi-page `from scratch' proofs, and `new' proofs keep coming.
We think that it is time to make a pause and summarize the main ideas behind the assertions.

In the case of the key `fixed $x$' type Proposition~\ref{T2.1} characterizing linear error bounds, the implication (b) \folgt (a) in part (i) of the above list as well as the necessary conditions in parts (ii) and (iii) follow immediately from the definitions.
They are
included for the completeness of the picture.
The main assertions are the sufficiency of condition (a), and implications (c) \folgt (b) and (d) \folgt (b).
They employ the following fundamental tools of variational analysis:
\begin{itemize}
\item
\emph{\EVP} (sufficiency of condition (a));
\item
\emph{sum rules} for respective subdifferentials (implications (c) \folgt (b) and (d) \folgt (b)).
\end{itemize}

The seemingly counter-intuitive fact that sufficient conditions for nonlinear error bounds can be deduced from those for the corresponding linear ones
was demonstrated by Corvellec et al. using, first, the `\emph{change-of-metric principle}' \cite{CorMot08, AzeCor14}, and then the `\emph{change-of-function}' approach \cite{AzeCor17} (see also \cite{NgaThe08,KruLopYanZhu19}).
Our presentation here largely follows the latter one.
We emphasise that throughout the paper the word `nonlinear' is used
in the conventional sense: `not necessarily linear'.

In the general nonlinear
setting, conventional sufficient local error bound conditions, besides slopes and subdifferentials, naturally involve variable coefficients $\varphi'(f(u))$ computed at appropriate points $u\in[f>0]$.
Several publications have appeared recently proving alternative nonlinear sufficient conditions with coefficients, which involve $\varphi'$ depending not on values of the function $f$ but on the distance $d(u,[f\le0])$; cf. \cite{CorMot08,ChaChe14,YaoZhe16,ZhaZhe19,KruLopYanZhu19}.
Such results usually assume certain monotonicity of $\varphi'$.
We show in Theorem~\ref{T2.5} that the alternative sufficient conditions are consequences of the conventional ones.
Observe that when $\varphi$ is linear (the conventional linear case), the coefficients are constant, and there is no difference between `conventional' and `alternative' conditions.
\sloppy

The structure of the paper is as follows.
In the next Section~\ref{S2}, we provide basic definitions and facts used throughout the paper.
In particular, we define \emph{subdifferential slopes} and introduce special collections of
slope operators.
Conventional linear error bound conditions are discussed in Section~\ref{S3}.
It contains a preliminary statement -- Proposition~\ref{T2.1} -- treating the case when $x$ in \eqref{D0-1} is fixed, and the general Theorem~\ref{T2.2}, the latter being an easy consequence of the first.
Both statements contain a condition, which has not been used in this type of statements earlier.
Conventional and alternative nonlinear error bound conditions are discussed in Sections~\ref{S4} and \ref{S5}, respectively.
We demonstrate that the conventional nonlinear conditions
are straightforward consequences of the corresponding linear ones,
while the alternative conditions are consequences of the conventional ones.
In Sections~\ref{S3.1} and \ref{S6}, we illustrate the sufficient and necessary conditions for nonlinear error bounds by applying them to characterizing nonlinear \emph{subregularity} of general \SVM s and \emph{calmness} of solution and level set mappings of canonically perturbed convex semi-infinite optimization
problems, respectively.

\section{Preliminaries} \label{S2}

Our basic notation is standard, see, e.g., \cite{RocWet98,Mor06.1,DonRoc14}.
Throughout the paper, $X$ and $Y$ are usually either metric or normed spaces.
We use the same notations $d(\cdot,\cdot)$ and $\|\cdot\|$ for distances and norms in all spaces.
Normed spaces are often treated as metric spaces with the distance determined by the norm in the usual way.
If not explicitly stated otherwise, products of metric or normed spaces are assumed equipped with the maximum distances or norms, e.g., $\|(x,y)\|:=\max\{\|x\|,\|y\|\}$, {$(x,y)\in X\times Y$.}
If $X$ is a normed space, its topological dual is denoted by $X^*$, while $\langle\cdot,\cdot\rangle$ denotes the bilinear form defining the pairing between the two spaces.
The open unit balls in the primal and dual spaces are denoted by $\B$ and $\B^*$, respectively, and $B_\delta(x)$ stands for the open ball with center $x$ and radius $\delta>0$.
Symbols $\R$, $\R_+$ and $\N$ denote the sets of all real numbers, all nonnegative real numbers and all positive integers, respectively.
Given an $\al\in\R$, we denote $\al_+:=\max\{0,\al\}$.

For a set $\Omega\subset X$, its boundary is denoted by $\bd\Omega$, while $i_\Omega$ denotes its
\emph{indicator function}, i.e. $i_\Omega(x)=0$ if $x\in \Omega$ and $i_\Omega(x)=+\infty$ if $x\notin \Omega$.
The distance from a point $x$ to $\Omega$ is
$d(x,\Omega):=\inf_{u \in \Omega}d(u,x)$, and we use the convention ${d(x,\emptyset) = +\infty}$.
For an extended-real-valued function $f:X\to\R_\infty$,
its domain and epigraph are defined,
respectively, by
$\dom f:=\{x \in X\mid {f(x) < +\infty}\}$
and
$\epi f:=\{(x,\alpha) \in X \times \mathbb{R}\mid {f(x) \le \alpha}\}$.
The inverse of $f$ (if it exists) is denoted by $f\iv$.

A set-valued mapping $F:X\rightrightarrows Y$ between two sets $X$ and $Y$ is a mapping, which assigns to every $x\in X$ a subset (possibly empty) $F(x)$ of $Y$.
We use the notations $\gph F:=\{(x,y)\in X\times Y\mid
y\in F(x)\}$ and $\dom\: F:=\{x\in X\mid F(x)\ne\emptyset\}$
for the graph and the domain of $F$, respectively, and $F^{-1} : Y\rightrightarrows X$ for the inverse of $F$.
This inverse (which always exists with possibly empty values at some points) is defined by $F^{-1}(y) :=\{x\in X \mid y\in F(x)\}$, $y\in Y$. Obviously $\dom F^{-1}=F(X)$.

For a
function $f:X\rightarrow\R_\infty$ on a metric space,
its \emph{slope} \cite{DegMarTos80} (also known as \emph{descending slope} and \emph{strong slope}; cf. \cite{AzeCor04,AzeCorLuc02,Iof00,Iof17}) and \emph{nonlocal slope} (cf. \cite{Sim91,NgaThe08,FabHenKruOut10, ChaChe14}) at $x\in\dom f$ are defined, respectively, by
\begin{align}\label{slopes}
|\nabla f|(x):=\limsup_{u\rightarrow x,\,u\ne x}\dfrac{ [f(x)-f(u)]_+}{d(u,x)} \AND
|\nabla f|^\diamond(x):=\sup\limits_{u\ne x}
\dfrac{[f(x)-f_+(u)]_+}{d(u,x)},
\end{align}
where the function $f_+:X\rightarrow\R_\infty$ is defined by $f_+(u):=(f(u))_+$ for all $u\in X$.
They characterize the maximal rate of descent of $f$ at $x$.
If $x\notin\dom f$, we set ${|\nabla f|(x)=|\nabla f|^\diamond(x):=+\infty}$.

If $X$ is a normed space, the \emph{Fr\'echet} and \emph{Clarke subdifferentials} of $f$ at $x\in\dom f$ are defined as
(cf. \cite{Kru03,Cla83,Mor06.1})
\begin{gather}\label{sdF}
\partial^F f(x):=\left\{x^*\in X^*\mid \liminf_{\substack{u\to x,\,u\ne x}} \dfrac{f(u)-f(x)-\langle x^*,u-x\rangle}{\|u-x\|}\ge 0\right\},
\\\label{sdC}
\partial^Cf(x):=\left\{x^*\in X^*\mid \langle x^*,z\rangle \le f^\circ(x,z)
\;\; \text{for all}\;\;
z\in X\right\},
\end{gather}
where 
\begin{align*}
	f^{\circ}(x;z):=\lim\limits_{\eps\downarrow0}
	\limsup_{\substack{(u,\al)\rightarrow(x,f(x))\\ f(u)\le\al,\, t\downarrow 0}} \inf_{\|z'-z\|<\eps}
	\dfrac{f(u+tz')-\al}{t}.
\end{align*}
is the \emph{Clarke--Rockafellar directional derivative} \cite{Roc79}
of $f$ at $x$ in the direction $z\in X$.
The last representation takes a simpler form when $f$ is Lipschitz continuous near $x$:
\begin{align*}
	f^{\circ}(x;z)=
	\limsup_{\substack{u\rightarrow x,\, t\downarrow 0}} 	\dfrac{f(u+tz)-f(u)}{t}.
\end{align*}
If $x\notin\dom f$, we set $\partial^F{f}(x)=\partial^C{f}(x):=\es$.
The sets \eqref{sdF} and \eqref{sdC} are closed and convex, and satisfy
$\partial^F{f}(x)\subset\partial^C{f}(x)$.
If $f$ is convex, they
reduce to the subdifferential in the sense of convex analysis:
\begin{gather*}
\partial{f}(x):= \left\{x^\ast\in X^\ast\mid
f(u)-f(x)-\langle{x}^\ast,x-x\rangle\ge 0 \;\; \text{for all} \;\; u\in X \right\}.
\end{gather*}

The above subdifferentials possess certain sum rules; cf. \cite{Roc79,Cla83,Fab89,Kru03,Mor06.1}.

\begin{lemma}
\label{SR}
Suppose $X$ is a normed space, $f_1,f_2:X\to\R \cup\{+\infty\}$, and $x\in\dom f_1\cap\dom f_2$.
\begin{enumerate}
\item
Let $f_1$ and $f_2$ be convex and $f_1$ be continuous at a point in $\dom f_2$.
Then
$$\partial(f_1+f_2)(x)=\sd f_1(x)+\partial f_2(x).$$
\item
Let
$f_1$ be Lipschitz continuous and $f_2$ be lower semicontinuous in a neighbourhood of $x$.
Then
$$\partial^C(f_1+f_2)(x)\subset\sd^C f_1(x) +\partial^Cf_2(x).$$
\item
Let $X$ be Asplund,
$f_1$ be Lipschitz continuous and $f_2$ be lower semicontinuous in a neighbourhood of $x$.
Then, for any $x^*\in\partial^F(f_1+f_2)(x)$ and $\varepsilon>0$, there exist $x_1,x_2\in X$ with $\|x_i-x\|<\varepsilon$, $|f_i(x_i)-f_i(x)|<\varepsilon$ $(i=1,2)$, such that
$$x^*\in\partial^Ff_1(x_1) +\partial^Ff_2(x_2)+\varepsilon\B^\ast.$$
\end{enumerate}
\end{lemma}

Recall that a Banach space is \emph{Asplund} if every continuous convex function on an open convex set is Fr\'echet differentiable on a dense subset \cite{Phe93}, or equivalently, if the dual of each 
separable subspace is separable.
We refer the reader to \cite{Phe93,Mor06.1} for discussions about and characterizations of Asplund spaces.
All reflexive, particularly, all finite dimensional Banach spaces are Asplund.

The following quantities are often used for characterizing error bounds in normed spaces:
\begin{gather}\label{sld}
|\sd^Ff|(x):=d(0,\partial^Ff(x))\AND
|\sd^Cf|(x):=d(0,\partial^Cf(x)).
\end{gather}
Following \cite{Sim91,FabHenKruOut10}, we call them (respectively, \Fr\ and Clarke) \emph{subdifferential slopes} of $f$ at $x$.
When $f$ is convex, we write simply $|\sd f|(x)$\footnote{In \cite{Sim91} this quantity is called the \emph{least slope} of $f$ at $x$}.

The next lemma collects several 
relations between various slopes; cf. \cite{Iof00,AzeCor02,AzeCor04,Aze06, FabHenKruOut10,FabHenKruOut12,AzeCor14,Kru15,Iof17}, which justify their usage in sufficient and necessary conditions for error bounds and determine the hierarchy of such conditions.

\begin{lemma}\label{P1.1}
Let $X$ be a metric space, $f:X\rightarrow\R_\infty$, and $x\in\dom f$.
\begin{enumerate}
\item
If $f(x)>0$, then
$|\nabla f|(x)\le |\nabla f|^\diamond(x)$.
\cnta
\end{enumerate}
Suppose $X$ is a normed space.
\begin{enumerate}
\cntb
\item
$|\nabla f|(x)\le|\partial^F f|(x)$.
\item
If $f$ is convex, then
$|\nabla f|(x)=|\partial f|(x)$.
Moreover, if $f(x)>0$, then
$|\nabla f|(x)=|\nabla f|^\diamond(x)$.
\item
If $X$ is a Banach space and $f$ is \lsc, then
$|\nabla f|(x)\ge|\partial^Cf|(x).$
\item
If $X$ is an Asplund space and $f$ is \lsc, then
$$|\nabla f|(x)\ge\liminf_{u\to x,\,f(u)\to f(x)} |\partial^Ff|(u).$$
\end{enumerate}
\end{lemma}

\begin{remark}\label{R1.2}
Assertions (i)--(iii) in Lemma~\ref{P1.1} are straightforward.
The more involved assertions (iv) and (v) are consequences of the sum rules in Lemma~\ref{SR} for the respective subdifferentials.
Observe that the fuzzy sum rule for \Fr\ subdifferentials in Lemma~\ref{SR}(iii) naturally translates into the `fuzzy' (due to the $\liminf$ operation) inequality in (v).
The latter estimate was first established (in a slightly more general setting of `abstract' subdifferentials) in \cite[Proposition~4.1]{AzeCor04} (see also \cite[Proposition~3.1]{Iof00}).
We have failed to find assertion (iv) explicitly formulated in the literature.
However, its proof only requires replacing the fuzzy sum rule in part (iii) of Lemma~\ref{SR} with the `exact' one in part (ii); cf.
the reasoning provided in \cite[Remark~6.1]{AzeCor14} to justify a similar fact involving the limiting subdifferentials.

Clarke subdifferentials in Lemma~\ref{P1.1}(iv) and all the other statements in this paper can be replaced with Ioffe's \emph{approximate $G$-subdifferentials} \cite{Iof17} as they possess a sum rule similar to the one in Lemma~\ref{SR}(ii); see \cite[Theorem~4.69]{Iof17}.
Moreover, it is clear that instead of Clarke subdifferentials in general Banach spaces and \Fr\ subdifferentials in Asplund spaces as in parts (iv) and (v) of Lemma~\ref{P1.1}, one can consider more general \emph{subdifferential pairs} \cite{AzeCorLuc02} (subdifferentials `trusted' on a given space \cite{Iof17}) with subdifferentials possessing a sum rule either in the exact or fuzzy form, respectively.
We use Clarke and Fr\'echet subdifferentials in this paper to keep the presentation simple.
\end{remark}

We recall chain rules for slopes, and Clarke and Fr\'echet subdifferentials; cf., e.g., \cite{Kru03,NgaThe09,YaoZhe16,AzeCor17,CuoKru20.2,CuoKru21.2}.

\begin{lemma}\label{L2}
Let $X$ be a metric space, $f:X\rightarrow\R_\infty$, $\varphi:\R\rightarrow\R_\infty$,
$x\in\dom f$ and $f(x)\in\dom\varphi$.
Suppose $\varphi$ is nondecreasing on $\R$ and differentiable at $f(x)$ with $\varphi'(f(x))>0$.
\begin{enumerate}
\item
$|\nabla(\varphi\circ f)|(x)=\varphi'(f(x))|\nabla f|(x).$
\item
Let $f(x)>0$, $\varphi(t)\le0$ if $t\le0$, $\varphi(t)>0$ if $t>0$, and $\varphi$ is differentiable on $]0,f(x)[$,
with
$\varphi'$ is nonincreasing (nondecreasing).
Then
$|\nabla(\varphi\circ f)|^\diamond(x)\ge\varphi'(f(x))|\nabla f|^\diamond(x)$ (${|\nabla(\varphi\circ f)|^\diamond(x)}\le\varphi'(f(x))|\nabla f|^\diamond(x)$).
\cnta
\end{enumerate}
Suppose $X$ is a normed space.
\begin{enumerate}
\cntb
\item
$\sd^F(\varphi{\circ}f)(x)= \varphi'(f(x))\sd^Ff(x)$.
\item
If $\varphi$ is strictly differentiable at $f(x)$, then $\sd^C(\varphi{\circ}f)(x)= \varphi'(f(x))\sd^Cf(x)$.
\end{enumerate}
\end{lemma}

\begin{proof}
Only assertion (ii) seems new.
If $f$ attains its minimum on $X$ at $x$, then, thanks to the monotonicity of $\varphi$, we have
$|\nabla({\varphi\circ f})|^\diamond(x)=|\nabla f|^\diamond(x)=0$.
Thanks to the assumptions on $\varphi$, we have $(\varphi\circ f)_+(u)={\varphi(f_+(u))}$ and, by the mean value theorem,
\begin{align*}
\dfrac{\varphi(f(x))-\varphi(f_+(u))}{d(u,x)} =\varphi'(\theta)\dfrac{f(x)-f_+(u)}{d(u,x)}
\end{align*}
for some $\theta\in]f_+(u),f(x)[$.
The claimed inequalities follow from the respective monotonicity assumptions on $\varphi'$.
For assertion (iii), we refer the reader to \cite[Proposition~2.1]{CuoKru20.2}.
\end{proof}

\if{
\red{$|\mathfrak{D} f|(u)$}
\red{$|\mathfrak{d} f|(u)$}
\red{$|\mathcal{D} f|(u)$}
\red{$|\mathbb{D} f|(u)$}

\red{$|\tilde{\nabla} f|(u)$}
\red{$|\breve{\nabla} f|(u)$}
\red{$|\hat{\nabla} f|(u)$}
\red{$|\widehat{\nabla} f|(u)$}
\red{$|\widetilde{\nabla} f|(u)$}
\red{$|\overline{\nabla} f|(u)$}

\red{$|\sd f|(u)$}
\red{$|\sd^F f|(u)$}
\red{$|\sd^C f|(u)$}
}\fi

When formulating sufficient error bound conditions, we use special collections of slope operators.
This allows us to combine several assertions into one.
The first collection is defined recursively as follows:
\begin{enumerate}
\item
$|\mathfrak{D}f|^\circ:=\{|\nabla f|\}$;
\item
if $X$ is Banach, then
$|\mathfrak{D}f|^\circ :=|\mathfrak{D}f|^\circ\cup\{|\sd^C f|\}$;
\item
if $X$ is Asplund, then
$|\mathfrak{D}f|^\circ:=|\mathfrak{D}f|^\circ\cup\{|\sd^F f|\}$.
\cnta
\end{enumerate}
The other two collections are variations of $|\mathfrak{D}f|^\circ$:
\begin{enumerate}
\cntb
\item
${|\mathfrak{D}f|}:=
|\mathfrak{D}f|^\circ\cup\{|\nabla f|^\diamond\}$;
\item
${|\mathfrak{D}f|^\dag}:=
|\mathfrak{D}f|\setminus\{|\sd^F f|\}$.
\end{enumerate}
Thus, if $X$ is an Asplund space, then $|\mathfrak{D}f|=\{|\nabla f|^\diamond,|\nabla f|,|\sd^C f|,|\sd^F f|\}$,
and if $X$ is a
Banach space and $f$ is convex, then
$|\mathfrak{D}f|=|\mathfrak{D}f|^\dag=\{|\nabla f|^\diamond,|\nabla f|,|\sd f|\}$.
The `full' set $|\mathfrak{D}f|$ is going to play the main role in the sufficient error bounds conditions below.
In some conditions, we also use the `truncated' sets $|\mathfrak{D}f|^\circ$ and $|\mathfrak{D}f|^\dag$, excluding the operators $|\nabla f|^\diamond$ and $|\sd^F f|$, respectively.
\if{We are going to use also a truncated set
\begin{gather*}
|\mathfrak{D} f|^\circ:=|\mathfrak{D} f|\setminus\{|\nabla f|^\diamond\}.
\end{gather*}
}\fi

\section{Conventional linear error bound conditions} \label{S3}

The next preliminary statement treats the case when $x$ in the definition \eqref{D0-1} of linear error bounds is fixed.
It contains all the main ingredients used in the general statement (Theorem~\ref{T2.2}), the latter being an easy consequence of the first.
The nonlinear error bound statements in the subsequent sections are also direct or indirect consequences of the next proposition.

\begin{proposition}\label{T2.1}
Suppose $X$ is a metric space, $f:X \rightarrow \mathbb{R}_\infty$, $x\in[f>0]$, and $\tau>0$.
\begin{enumerate}
\item
Let $X$ be complete, $f$ be lower semicontinuous, and ${\al\in]0,1]}$.
The error bound inequality \eqref{D0-1}
holds at $x$, provided that one of the following conditions is satisfied:
\begin{enumerate}
\item
$|\breve\nabla f|\in|\mathfrak{D}f|^\dag$ and
$\al|\breve\nabla f|(u)\ge\tau$ for
all $u\in X$ satisfying
\begin{gather}\label{T2.1-2}
f(u)\le f(x),
\\\label{T2.1-3}
d(u,x)<\alpha d(x,[f\le0]),
\\\label{T2.1-4}
\al f(u)<\tau d(u,[f\le0]),
\\\label{T2.1-5}
f(u)<\tau d(x,[f\le0]);
\end{gather}

\item
$X$ is Asplund and there exists a $\mu>f(x)$ such that $\al|\sd^F f|(u)\ge\tau$
for all $u\in X$ satisfying $f(u)<\mu$, and conditions \eqref{T2.1-3}--\eqref{T2.1-5}.
\end{enumerate}

\item
If the error bound inequality \eqref{D0-1}
holds, then
$|\nabla f|^\diamond(x)\ge\tau$.
\item
Let $X$ be a normed space, and $f$ be convex.
If the error bound inequality \eqref{D0-1}
holds, then
$|\partial f|(x)\ge\tau$.
\end{enumerate}
\end{proposition}

\if{
\AK{30/09/20.
Can condition (d) be weakened? I am thinking about something like this:\\
$X$ is Asplund and, for any $\theta\in]0,1[$, there exists a $\mu>f(x)$ such that $d(0,\partial^Ff(u))\ge\al\iv\tau\theta$
for all $u\in X$ satisfying $f(u)<\mu$, and conditions \eqref{T2.1-3} and \eqref{T2.1-4}.
}
\NDC{19/11/20
I think the condition can be weakened as you suggested since we have the flexibility of choosing $\varepsilon$ in the Asplund space.
At the same time, it seems to me that the weakened condition is not very meaningful since it requires the metric subdifferential inequality hold for all $\theta\in]0,1[$.
I meant when checking the error bound, we have to use $\theta=1$.}
}\fi

All but one arguments in the short proof below have been used many times in numerous proofs of this type of assertions.

\begin{proof}
\begin{enumerate}
\item
Suppose that the error bound inequality \eqref{D0-1} does not
hold at $x$, i.e.
\begin{align}\label{T2.1P1}
f(x)<\tau d(x,[f \le 0]).
\end{align}
Choose a $\tau'\in]0,\tau[$ such that $f(x)<\tau'd(x,[f \le 0]).$
By the \EVP\ applied to the \lsc\ function $f_+:=\max\{f,0\}$, there exists a point $u\in X$ satisfying \eqref{T2.1-2} and \eqref{T2.1-3}, and such that
\begin{align}\label{T2.1P2}
f_+(u)\le f_+(u')+\al\iv\tau'd(u',u) \qdtx{for all}u'\in X.
\end{align}
We show that $u$ satisfies also \eqref{T2.1-4} and \eqref{T2.1-5}, while $\al|\nabla f|^\diamond(u)<\tau$.
By \eqref{T2.1-3}, $u\notin[f\le0]$.
Hence, $f_+(u)=f(u)$, and it follows from \eqref{T2.1P2} that
$\al|\nabla f|^\diamond(u)\le\tau'<\tau$, and $\al f(u)\le\tau'd(u',u)$ for all $u'\in[f\le0]$.
The last inequality yields \eqref{T2.1-4}, while \eqref{T2.1P1} and \eqref{T2.1-2} imply \eqref{T2.1-5}.
This proves the sufficiency of condition~(a) with $|\breve\nabla f|=|\nabla f|^\diamond$.
The sufficiency of this condition with the other components of $|\mathfrak{D}f|^\circ$ and the implication (b) \folgt (a) are consequences of Lemma~\ref{P1.1}(i), (iv) and~(v).

\item
is an immediate consequence of the definition of the nonlocal slope  in \eqref{slopes}.
\item
follows from (ii) thanks to Lemma~\ref{P1.1}(iii).
\end{enumerate}
\end{proof}

\begin{remark}\label{R2.1--}
\begin{enumerate}
\item
In view of the definition of $|\mathfrak{D}f|^\dag$,
condition~(a) in Proposition~\ref{T2.1}(i) combines three separate primal and dual sufficient error bound conditions:
\begin{enumerate}
\item[(a1)]
$\al|\nabla f|^\diamond(u)\ge\tau$
for all $u\in X$ satisfying conditions \eqref{T2.1-2}--\eqref{T2.1-5};
\item[(a2)]
$\al|\nabla f|(u)\ge\tau$
for all $u\in X$ satisfying conditions \eqref{T2.1-2}--\eqref{T2.1-5};
\item[(a3)]
$X$ is Banach and $\al|\sd^Cf|(u)\ge\tau$
for all $u\in X$ satisfying conditions \eqref{T2.1-2}--\eqref{T2.1-5}.
\end{enumerate}
Moreover, thanks to parts (i), (iv) and~(v) of Lemma~\ref{P1.1}, we have {\rm (a3) \folgt (a2) \folgt (a1)} and {\rm (b) \folgt (a2)}.
Thus, condition (a) in Proposition~\ref{T2.1}(i) can be replaced equivalently with the simpler condition (a1), the weakest of the three sufficient conditions above.
\if{
Note that condition (a1) is nonlocal and, hence, not very convenient for applications.
The majority of authors prefer local conditions of the type (a2), (a3) or (b).
Such conditions are often established independently as separate assertions.
Combining them in a single statement makes the presentation more succinct.
}\fi
Conditions (a1), (a2), (a3) and (b) represent four types of sufficient error bound conditions frequently appearing in the literature, with
each of them having its own area of applicability.
Such conditions are often proved independently as separate assertions.
Proposition~\ref{T2.1}(i) seems to be the first attempt to combine them in a single statement.

With obvious minor adjustments, this observation applies to all assertions in this paper containing multi-component collections of slope operators $|\mathfrak{D}f|^\circ$, $|\mathfrak{D}f|^\dag$ or $|\mathfrak{D}f|$.

\item
The core of Proposition~\ref{T2.1}(i) is made of the sufficiency of condition (a1), which is a consequence of the \EVP, and implications (a3) \folgt (a2) and (b)~\folgt (a2), which follow from the sum rules for respective subdifferentials.
\item
The parameter $\al$ in Proposition~\ref{T2.1}(i) arises naturally from the application of the \EVP.
In most cases this type of assertions are formulated with $\al=1$.
Taking a smaller $\al$, strengthens the slope
inequalities in the sufficient conditions at the expense of reducing the set of points satisfying inequality \eqref{T2.1-3}.
This `trade-off' parameter has been used in several publications
\cite{YaoZhe16,ZheZhu16,KruLopYanZhu19,ZhaZhe19}.
\item
Restrictions \eqref{T2.1-2}--\eqref{T2.1-5} on the choice of $u\in X$ and inequality $f(u)<\mu$ in condition (b) in Proposition~\ref{T2.1}(i)
also arise naturally from the application of the \EVP.
Weakening or dropping any/all of these restrictions produces new (stronger!) sufficient conditions widely used in the literature.
This can be particularly relevant in the case of restrictions \eqref{T2.1-3}--\eqref{T2.1-5} involving distances to the unknown set $[f\le0]$.
Note that, if restriction \eqref{T2.1-3} is dropped, it makes sense checking the resulting sufficient condition with $\al=1$.
If a point $\bx\in[f\le0]$ is known, restrictions \eqref{T2.1-3}, \eqref{T2.1-4} or \eqref{T2.1-5} can be replaced, respectively, with the weaker inequalities:
\begin{gather*}
d(u,x)<\alpha d(x,\bx),\quad
\al f(u)<\tau d(u,\bx)
\qdtx{or}
f(u)<\tau d(x,\bx).
\end{gather*}
\item
This seems to be the first time that
inequality \eqref{T2.1-5} appears as a part of sufficient linear error bound conditions.
It is a bit surprising since its nonlinear analogues have been exploited in the literature; see, e.g., \cite{ZhaZhe19}.
We demonstrate in the next theorem that this inequality can be meaningful in the linear setting too, thus, paving the way to the subsequent nonlinear extensions.
\item
Since $\partial^F{f}(x)\subset\partial^C{f}(x)$, the subdifferential slope inequality in condition (a3)
obviously implies the one in condition (b).
However, the implication (a3)~\folgt (b) is not true in general because the restriction $f(u)<\mu$ in (b) is weaker than the corresponding inequality \eqref{T2.1-2} in (a).
The number $\mu>f(x)$ in condition (b) can be chosen arbitrarily close to $f(x)$, but cannot be replaced with $f(x)$ because of the `fuzzy' inequality in Lemma~\ref{P1.1}(v), which, in turn, is a consequence of the fuzzy sum rule for \Fr\ subdifferentials in Lemma~\ref{SR}(iii).
\end{enumerate}
\end{remark}

The general error bound statement in the next theorem is a straightforward consequence of Proposition~\ref{T2.1}.
In accordance with the convention made in the Introduction after Definition~\ref{D0}, we talk here about the function $f$ admitting a $\tau-$error bound at $\bx$ `with $\delta$ (or $\delta'$) and~$\mu$'.

\begin{theorem}\label{T2.2}
Suppose $X$ is a metric space, $f:X \rightarrow \mathbb{R}_\infty$, $\bx\in X$, $\tau>0$, $\de\in]0,+\infty]$ and $\mu\in]0,+\infty]$.
\begin{enumerate}
\item
Let $X$ be complete, $f$ be lower semicontinuous, $\al\in]0,1]$, and either $\bx\in[f\le0]$ or $\de=+\infty$.
Let $|\breve\nabla f|\in|\mathfrak{D}f|$.
The function $f$ admits a $\tau-$error bound at $\bx$ with $\delta':=\frac{\de}{1+\al}$ and $\mu$,
provided that
$\al|\breve\nabla f|(u)\ge\tau$
for all $u\in B_\delta(\bx)\cap[0<f<\mu]$ satisfying
\begin{align}\label{T2.2-2}
\max\{\al,1-\al\}f(u)<\tau d(u,[f\le0]).
\end{align}

\item
If $f$ admits a $\tau-$error bound at $\bx$ with $\delta$ and $\mu$, then
$|\nabla f|^\diamond(u)\ge\tau$ for all $u\in B_\de(\bx)\cap[0<f<\mu]$.
\item
Let $X$ be a normed space, and $f$ be convex.
If $f$ admits a $\tau-$error bound at $\bx$ with $\delta$ and $\mu$, then
$|\partial f|(u)\ge\tau$ for all $u\in B_\de(\bx)\cap[0<f<\mu]$.
\if{
Moreover, if
$\tau d(u,[f\le0])=f(u)$ for all
$u\in B_\de(\bx)\cap[0<f<\mu]$, then $d(0,\partial f(u))=\tau$ for all $u\in B_\de(\bx)\cap[0<f<\mu]$.
}\fi
\end{enumerate}
\end{theorem}

\if{
\NDC{26.7.20
The `moreover' part in item (iii) is motivated by \cite[Propsition~4.2]{CorMot08}.}

\AK{27.7.20.
I am not sure if this is important.}
\begin{proof}
(iii)
\red{
Let $u\in B_\de(\bx)\cap [0<f<\mu]$.
Under the assumption made, we have
$\tau d(u,[f\le0])=f(u)$.
Then $d(0,\partial f(u))=|\nabla f|(u)\ge\tau$ thanks to  Proposition~\ref{T2.1}(ii) and the convexity of $f$.
Take an arbitrarily $u\in B_\de(\bx)$.
If $f(u)\ge f(u)$,  then $\frac{[f(x)-f(u)]_+}{d(x,u)}=0$.
If $f(u)<f(x)$, then
\begin{align*}
\dfrac{f(x)-f(u)}{d(x,u)}
=\tau.\dfrac{d(x,[f\le 0])-d(u,[f\le 0])}{d(x,u)}\le\tau,
\end{align*}
thanks to the 1-Lipschitz property of the distance function.
Hence, it always holds that $|\nabla f|(x)\le\tau$, and consequently, $|\nabla f|(x)=d(0,\partial f(x))=\tau$.
The proof is complete.
}
\end{proof}
}\fi

\begin{proof}
To prove assertion (i),
it suffices to check that, if $x\in B_{\de'}(\bx)\cap[0<f<\mu]$ and $u\in X$ satisfies conditions \eqref{T2.1-3}--\eqref{T2.1-5} and $f(u)<\mu$ (in particular if it satisfies \eqref{T2.1-2}), then $u\in B_\delta(\bx)\cap[0<f<\mu]$, and it satisfies conditions \eqref{T2.2-2}.
Indeed, we have $u\in[f<\mu]$.
In view of \eqref{T2.1-3}, $u\in[f>0]$ and $d(u,x)<\al d(x,\bx)$.
The last inequality together with $x\in B_{\de'}(\bx)$ yield $u\in B_{\de}(\bx)$.
Condition \eqref{T2.1-3} obviously implies $d(u,[f\le0])>(1-\al)d(x,[f\le0])$, and it follows from \eqref{T2.1-5} that $(1-\al)f(u)<\tau d(u,[f\le0])$.
Together with \eqref{T2.1-4}, this gives \eqref{T2.2-2}.
Assertions (ii) and (iii) follow immediately from the corresponding assertions in Proposition~\ref{T2.1}.
\end{proof}

\begin{remark}\label{R2.1-}
\begin{enumerate}
\item
In view of the definition of $|\mathfrak{D}f|$,
Theorem~\ref{T2.2}(i) combines four separate primal and dual sufficient error bound conditions corresponding to $|\breve\nabla f|$ equal to $|\nabla f|^\diamond$, $|\nabla f|$, $|\sd^Cf|$ or $|\sd^Ff|$ (in appropriate spaces).
\item
The parameter $\al$ in Theorem~\ref{T2.2}(i) determines a trade-off between the main inequality $\al|\breve\nabla f|(u)\ge\tau$ (and also inequality \eqref{T2.2-2}) and the radius $\de'$ of the \nbh\ of $\bx$ in which the error bound estimate holds; cf. Remark~\ref{R2.1--}(iii).
In the conventional case $\al=1$, we have $\de'=\de/2$ as it has been observed in numerous publications.

\item
Weakening or dropping any of the restrictions on $u$ produces new (stronger!) sufficient conditions; cf. Remark~\ref{R2.1--}(iv).
This can be particularly relevant in the case of inequality \eqref{T2.2-2} involving the distance to the unknown set $[f\le0]$.
If $\bx\in[f\le0]$,
it is
common to replace this distance with $d(x,\bx)$.

\item
Sufficient error bound conditions of the type in Theorem~\ref{T2.2}(i) with the weaker inequality $\al f(u)<\tau d(u,[f\le0])$ in place of \eqref{T2.2-2} can be found in the literature (cf., e.g., \cite[Theorem 3.4]{KruLopYanZhu19}).
The fact that this inequality can be strengthened by replacing $\al$ with $\max\{\al,1-\al\}$ seems to be observed for the first time here.
It is a consequence of condition \eqref{T2.1-5} in Proposition~\ref{T2.1}.

\item
Under the conditions of part (iii) of Theorem~\ref{T2.2},
one can easily show that,
if for all ${x\in B_\de(\bx)}\cap[0<f<\mu]$ condition \eqref{D0-1}
holds as equality, then
$|\partial f|(x)=\tau$ for all ${x\in B_\de(\bx)}\cap[0<f<\mu]$.
\end{enumerate}
\end{remark}

The 
local
$\tau-$error bound
conditions are collected in the next three corollaries.

\begin{corollary}\label{C2.1}
Suppose $X$ is a complete metric space, $f:X \rightarrow \mathbb{R}_\infty$ is \lsc, and $\tau>0$.
Let $|\breve\nabla f|\in|\mathfrak{D}f|$.
The function $f$ admits a local $\tau-$error bound at $\bx\in[f\le0]$, provided that
$|\breve\nabla f|(x)\ge\tau$
for all $x\in[f>0]$ near $\bx$ with $f(x)$ near $0$.
Moreover, if ${|\breve\nabla f|=|\nabla f|^\diamond}$, then the above condition is also necessary.
\end{corollary}


\begin{corollary}
Suppose $X$ is a Banach space, $f:X \rightarrow \mathbb{R}_\infty$ is convex \lsc,  and $\tau>0$.
Let $|\breve\nabla f|\in|\mathfrak{D}f|$.
The function $f$ admits a local $\tau-$error bound at $\bx\in[f\le0]$ if and only if
$|\breve\nabla f|(x)\ge\tau$
for all $x\in[f>0]$ near $\bx$.
\end{corollary}

\begin{corollary}
Suppose $X$ is a complete metric space, $f:X \rightarrow \mathbb{R}_\infty$ is \lsc, and $\bx\in[f\le0]$. Then
\begin{align*}
\Er f(\bx) =\liminf_{x\to\bx,\,f(x)\downarrow0}|\nabla f|^\diamond(x) \ge\liminf_{x\to\bx,\,f(x)\downarrow0}|\nabla f|(x).
\end{align*}
If $X$ is Banach (Asplund), then $|\nabla f|$ in the above inequality can be replaced with $|\partial^Cf|$ ($|\partial^Ff|$).
If $X$ is Banach and $f$ is convex, then
\begin{align*}
\Er f(\bx)
=\liminf_{x\to\bx,\,f(x)>0}|\partial f|(x).
\end{align*}
\end{corollary}

\begin{remark}
The limits
\begin{align}\label{R2.3-1}
\liminf_{x\to\bx,\,f(x)\downarrow0}|\nabla f|^\diamond(x),\quad \liminf_{x\to\bx,\,f(x)\downarrow0}|\nabla f|(x)\AND \liminf_{x\to\bx,\,f(x)\downarrow0}|\partial f|(x)
\end{align}
are referred to in \cite{FabHenKruOut10,FabHenKruOut12,Kru15} as, respectively, the \emph{strict outer}, \emph{uniform strict outer} and \emph{strict outer subdifferential} slopes of $f$ at $\bx$; cf. \emph{limiting} slopes \cite{Iof00,Iof17}.
\end{remark}


The error bound inequalities \eqref{D0-1}, \eqref{D1-1} and \eqref{D1-2} correspond to the sublevel set $[f\le0]$.
Definitions~\ref{D0} and \ref{D1} and the corresponding error bound conditions can be easily extended to the case of an arbitrary sublevel set $[f\le c]$ where $c\in\R$.
It suffices to replace $f$ in the definitions and statements with $f-c$ (with the corresponding small adjustment in the definition of the nonlocal slope).

The nonlocal slope $|\nabla f|^\diamond(x)$ in Proposition~\ref{T2.1}(ii) cannot in general be replaced with the local one unless $f$ is convex.
As observed in \cite[proof of Proposition~2.1]{AzeCor04}, this can be done if instead of the fixed error bound inequality \eqref{D0-1} one considers a family of perturbed ones.

\begin{proposition}\label{P2.1}
Suppose $X$ is a metric space, $f:X \rightarrow \mathbb{R}_\infty$, $x\in [f>0]$, and $\tau>0$.
If
\begin{align}\label{P2.1-1}
\tau d(x,[f \le c])\le f(x)-c
\end{align}
for all sufficiently large $c<f(x)$, then
$|\nabla f|(x)\ge\tau$.
\end{proposition}

Theorem~\ref{T2.2} and Proposition~\ref{P2.1} yield the following statement for `perturbed' error bounds
extending \cite[Theorem~2.1]{AzeCor04},
\cite[Theorem~2.3]{CorMot08} and \cite[Theorem~3.2]{AzeCor17}.

\begin{proposition}\label{P2.2}
Suppose $X$ is a metric space, $f:X \rightarrow \mathbb{R}_\infty$, $\bx\in X$, $\tau>0$, ${\de\in]0,+\infty]}$ and $\mu\in]0,+\infty]$.
\begin{enumerate}
\item
Let $X$ be complete, $f$ be lower semicontinuous, $\al\in]0,1]$, and either $\bx\in[f\le0]$ or $\de=+\infty$.
Let $|\breve\nabla f|\in|\mathfrak{D}f|^\circ$.
The perturbed error bound inequality \eqref{P2.1-1} holds
for all $c\in[0,\mu[$
{and}
$x\in B_{\frac{\de}{1+\al}}(\bx)\cap[c<f<\mu]$
provided that
$\al|\breve\nabla f|(x)\ge\tau$
for all $x\in B_\delta(\bx)\cap[0<f<\mu]$.
\item
If the perturbed error bound inequality \eqref{P2.1-1}
holds for all $c\in[0,\mu[$ and $x\in B_\de(\bx)\cap[{c<f<\mu}]$, then $|\nabla f|(x)\ge\tau$
for all $x\in B_\delta(\bx)\cap[0<f<\mu]$.

\end{enumerate}
\end{proposition}

\begin{remark}\label{R2.3}
\begin{enumerate}
\item
Proposition~\ref{P2.2} provides necessary and sufficient conditions for the following \emph{perturbed $\tau-$error bound} property of $f$ at $\bx\in X$ with some $\delta\in]0,+\infty]$ and $\mu\in]0,+\infty]$ such that either $\bx\in[f\le0]$ or $\de=+\infty$: inequality \eqref{P2.1-1}
holds for all $c\in[0,\mu[$ and $x\in B_\de(\bx)\cap[c<f<\mu]$.
\item
Thanks to Corollary~\ref{C2.1} and Proposition~\ref{P2.2}, conditions
\begin{itemize}
\item
$|\nabla f|^\diamond(x)\ge\tau$
for all $x\in[f>0]$ near $\bx$ with $f(x)$ near $0$, and
\item
$|\nabla f|(x)\ge\tau$
for all $x\in[f>0]$ near $\bx$ with $f(x)$ near $0$
\end{itemize}
provide full characterizations of, respectively, the $\tau-$error bound and the perturbed $\tau-$error bound properties of $f$ at $\bx\in[f\le0]$.
In view of Lemma~\ref{P1.1}(iii),
in the convex case the perturbed $\tau-$error bounds are equivalent to the conventional ones.
\item
The perturbed $\tau-$error bound property of $f$ at $\bx\in[f\le0]$ is actually the $\tau-$\emph{metric regularity} of the (truncated) \emph{epigraphical} \SVM\ $x\mapsto\epi f(x):=\{{c\in[0,+\infty[}\,\mid f(x)\le c\}$ at $(\bx,0)$; cf. \cite[Remark~3.2]{AzeCor17}.
\end{enumerate}
\end{remark}

\section{Nonlinear error bound conditions} \label{S4}

In view of Remark~\ref{R1.1}(ii), one can easily deduce from Theorem~\ref{T2.2} and Proposition~\ref{T2.1} sufficient and necessary conditions for nonlinear error bounds.
The sufficient conditions become meaningful when $\varphi\in\mathcal{C}^1$ as in this case one can employ the chain rules in Lemma~\ref{L2}.

\begin{theorem}\label{T2.4}
Suppose $X$ is a metric space, $f:X \rightarrow \mathbb{R}_\infty$, $\bx\in X$, ${\varphi\in\mathcal{C}^1}$, $\de\in]0,+\infty]$ and $\mu\in]0,+\infty]$.
\begin{enumerate}
\item
Let $X$ be complete, $f$ be lower semicontinuous, $\al\in]0,1]$, and either $\bx\in[f\le0]$ or $\de=+\infty$.
The function $f$ admits a $\varphi-$error bound at $\bx$ with $\delta':=\frac{\de}{1+\al}$ and $\mu$,
provided that one of the following conditions is satisfied:
\begin{enumerate}
\item
$\al|\nabla(\varphi\circ f)|^\diamond(u)\ge1$
for all $u\in B_\delta(\bx)\cap[0<f<\mu]$ satisfying
\begin{align}\label{T2.4-1}
\max\{\al,1-\al\}\varphi(f(u))<d(u,[f\le0]);
\end{align}
\item
$|\breve\nabla f|\in|\mathfrak{D}f|^\circ$ and $\al\varphi'(f(u))|\breve\nabla f|(u)\ge1$
for all $u\in B_\delta(\bx)\cap[0<f<\mu]$ satisfying condition~\eqref{T2.4-1}.
\end{enumerate}
If $\varphi'$ is nonincreasing, then $|\mathfrak{D}f|^\circ$ in {\rm (b)} can be replaced with $|\mathfrak{D}f|$.

\item
If $f$ admits a $\varphi-$error bound at $\bx$ with $\delta$ and $\mu$, then
$|{\nabla(\varphi\circ f)}|^\diamond(u)\ge1$ for all ${u\in B_\de(\bx)}\cap[0<f<\mu]$.
\item
Let $X$ be a normed space, and $f$ be convex.
If $f$ admits a $\varphi-$error bound at $\bx$ with $\delta$ and $\mu$, then
$\frac{\varphi(f(u))}{f(u)}|\partial f|(u)\ge1$ for all $u\in B_\de(\bx)\cap[0<f<\mu]$.\\
If, moreover, $\varphi'$ is nondecreasing, particularly if $\varphi$ is convex, then
$\varphi'(f(u))|\partial f|(u)\ge1$ for all $u\in B_\de(\bx)\cap[0<f<\mu]$.
\end{enumerate}
\end{theorem}

\begin{proof}
Assertions (i) and (ii)
are direct consequences of the corresponding assertions in Theorem~\ref{T2.2}, applied to the composition function $\varphi\circ f$ with $\tau=1$, and Lemma~\ref{L2}.
To prove the first part of assertion (iii), it suffices to notice that inequality \eqref{D0-1} reduces to inequality \eqref{D1-1} by setting
$\tau:=\frac{f(x)}{\varphi(f(x))}$
and apply Proposition~\ref{T2.1}(iii).
By the mean value theorem, $\tau\iv=\varphi'(\theta)$
for some $\theta\in]0,f(x)[$.
If $\varphi'$ is nondecreasing, then $\tau\iv\le\varphi'(f(x))$, which proves the second part.
\end{proof}

\begin{remark}
\begin{enumerate}
\item
In view of the definition of $|\mathfrak{D}f|^\circ$ (or $|\mathfrak{D}f|$), condition (b) in
Theorem~\ref{T2.4}(i) combines three (or four) separate primal and dual sufficient error bound conditions corresponding to $|\breve\nabla f|$ equal to $|\nabla f|^\diamond$, $|\nabla f|$, $|\sd^Cf|$ or $|\sd^Ff|$ (in appropriate spaces).
\item
With $|\breve\nabla f|=|\sd^Cf|$ or $|\breve\nabla f|=|\sd^Ff|$, inequality $\al\varphi'(f(u))|\breve\nabla f|(u)\ge1$ in (b) can be interpreted as the Kurdyka--{\L}ojasiewicz property (as defined, e.g., in \cite{AttBolRedSou10,BolDanLeyMaz10}).
\item
As in the linear case, parameter $\al$ in Theorem~\ref{T2.4}(i) determines a trade-off between the main inequalities $\al|\nabla(\varphi\circ f)|^\diamond(u)\ge1$ in (a) and $\al\varphi'(f(u))|\breve\nabla f|(u)\ge1$ in (b) and the radius $\de'$ of the \nbh\ of $\bx$ in which the error bound estimate holds.
In the conventional case $\al=1$, we have $\de'=\de/2$.
\item
Weakening or dropping any of the restrictions on $u$ produces new (stronger!) sufficient conditions.
This can be particularly relevant in the case of inequality \eqref{T2.4-1} involving the distance to the unknown set $[f\le0]$.
If $\bx\in[f\le0]$,
it is
common to replace this distance with $d(u,\bx)$.
\item
If the function $t\mapsto\frac{\varphi(t)}{t}$ is nondecreasing on $]0,+\infty[$ (particularly, if $\varphi$ is convex) then the assumption of differentiability of $\varphi$ in Theorem~\ref{T2.4} can be dropped.
As one can observe from the above proof, it suffices to replace $\varphi'(f(u))$ in condition (b) with $\frac{\varphi(f(u))}{f(u)}$.

\item
Under the conditions of part (iii) of Theorem~\ref{T2.4},
one can easily show that,
if for all $u\in B_\de(\bx)\cap[0<f<\mu]$ condition \eqref{D1-1}
holds as equality, then
$\varphi'(f(u))|\partial f|(u)=1$ for all $u\in B_\de(\bx)\cap[0<f<\mu]$; cf. Remark~\ref{R2.1-}(v).
\sloppy
\item
Employing Proposition~\ref{P2.2}, one can
expand Theorem~\ref{T2.4} to cover a \emph{perturbed $\varphi-$error bound} property of $f$ at $\bx\in[f\le0]$
as in \cite[Theorems~4.1 and 4.2]{AzeCor17}; cf. Remark~\ref{R2.3}.
\item
With $\de=+\infty$, $\mu<+\infty$, $\al=1$ and $\varphi(t):=\tau\iv t^q$ for some $\tau>0$ and $q>0$ and all $t>0$ (H\"older case), Theorem~\ref{T2.4}(i) with $|\breve\nabla f|=|\nabla f|$ in condition (b) recaptures \cite[Corollary~2.5]{NgaThe08},
while with $|\breve\nabla f|=|\sd^Ff|$ it recaptures \cite[Corollary~2(i)]{NgaThe09}.
\item
With $\bx\in\bd[f\le0]$, $\de<+\infty$, $\mu=+\infty$, $\al=1$, and $\varphi(t):=\tau\iv t^q$ for some $\tau>0$ and $q>0$ and all $t>0$ (H\"older case), Theorem~\ref{T2.4}(i) with $|\breve\nabla f|=|\sd^Ff|$ in condition (b) recaptures \cite[Corollary~2(ii)]{NgaThe09}.
\item
With $\bx\in[f\le0]$, $\mu=+\infty$, and $\varphi(t):=(\al\tau)\iv t^q$ for some $\tau>0$ and $q>0$ and all $t>0$ (H\"older case), part (i) of Theorem~\ref{T2.4} with $|\breve\nabla f|=|\sd^Cf|$ and $|\breve\nabla f|=|\sd^Ff|$ in condition (b) improves \cite[Theorem~3.7]{KruLopYanZhu19}, while part (iii) partially recaptures and extends \cite[Lemma~3.33]{KruLopYanZhu19}.
\end{enumerate}
\end{remark}

The
local
$\varphi-$error bound sufficient conditions are collected in the next three corollaries.

\begin{corollary}\label{C2.3}
Suppose $X$ is a complete metric space, $f:X \rightarrow \mathbb{R}_\infty$ is \lsc, and ${\varphi\in\mathcal{C}^1}$.
The function $f$ admits a local $\varphi-$error bound at $\bx\in[f\le0]$, provided that one of the following conditions is satisfied:
\begin{enumerate}
\item
$|\nabla(\varphi\circ f)|^\diamond(x)\ge1$
for all $x\in[f>0]$ near $\bx$ with $f(x)$ near $0$;

\item
$|\breve\nabla f|\in|\mathfrak{D}f|^\circ$ and $\varphi'(f(x))|\breve\nabla f|(x)\ge1$
for all $x\in[f>0]$ near $\bx$ with $f(x)$ near $0$.
\end{enumerate}
Condition {\rm (i)} is also necessary.\\
If $\varphi'$ is nonincreasing, then $|\mathfrak{D}f|^\circ$ in {\rm (ii)} can be replaced with $|\mathfrak{D}f|$.
\end{corollary}

\begin{corollary}\label{C2.4}
Suppose $X$ is a Banach space, $f:X \rightarrow \mathbb{R}_\infty$ is convex \lsc, $\bx\in[f\le0]$, and ${\varphi\in\mathcal{C}^1}$.
Consider the following conditions:
\begin{enumerate}
\item
$f$ admits a local $\varphi-$error bound at $\bx$;

\item
$|\nabla(\varphi\circ f)|^\diamond(x)\ge1$
for all $x\in[f>0]$ near $\bx$ with $f(x)$ near $0$;


\item
$\varphi'(f(x))|\partial f|(x)\ge1$ for all $x\in[f>0]$ near $\bx$ with $f(x)$ near $0$;

\item
$\frac{\varphi(f(x))}{f(x)}|\partial f|(x)\ge1$ for all $x\in[f>0]$ near $\bx$ with $f(x)$ near $0$.
\end{enumerate}
Then {\rm(iii) \folgt (ii) \folgt (i) \folgt (iv)}.
Moreover, if $\varphi'$ is nondecreasing, particularly if $\varphi$ is convex, then all the conditions are equivalent.
\end{corollary}

\begin{corollary}\label{C2.6--}
Suppose $X$ is a complete metric space, $f:X \rightarrow \mathbb{R}_\infty$ is \lsc,  ${\varphi\in\mathcal{C}^1}$, $\bx\in[f\le0]$, and $\Er_\varphi f(\bx)$ be defined by \eqref{D0-3}.
\begin{enumerate}
\item
The following estimate holds true:
\begin{gather*}
\Er_\varphi f(\bx)=
\liminf_{x\to\bx,\,f(x)\downarrow0}|\nabla(\varphi\circ f)|^\diamond(x)\ge
\liminf_{x\to\bx,\,f(x)\downarrow0}\varphi'(f(x))|\nabla f|(x).
\end{gather*}
If $X$ is Banach (Asplund), then $|\nabla f|$ in the above inequality can be replaced with $|\partial^Cf|$ ($|\partial^Ff|$).

\item
If $X$ is Banach and $f$ is convex, then
\begin{gather*}
\Er_\varphi f(\bx)\le
\liminf_{x\to\bx,\,f(x)\downarrow0} \frac{\varphi(f(x))}{f(x)}|\partial f|(x).
\end{gather*}
Moreover, if $\varphi$ satisfies
\begin{align}\label{C2.6-2}
\frac{\varphi(t)}{t}\le\ga\varphi'(t)\qdtx{for some}\ga\ge1 \qdtx{and all}t>0,
\end{align}
then $\overline{|\sd f|}{}_\varphi^>(\bx)\le\Er_\varphi f(\bx)
\le\ga\overline{|\sd f|}{}_\varphi^>(\bx)$, where
\begin{align}\label{C2.6-3}
\overline{|\sd f|}{}_\varphi^>(\bx) :=\liminf_{x\to\bx,\,f(x)\downarrow0}\varphi'(f(x))|\sd f|(x);
\end{align}
as a consequence, $f$ admits a local $(\al\varphi)-$error bound at $\bx$ with some $\al$ satisfying $\overline{|\sd f|}{}_\varphi^>(\bx)\le\al\iv
\le\ga\overline{|\sd f|}{}_\varphi^>(\bx)$ if and only if $\overline{|\sd f|}{}_\varphi^>(\bx)>0$.
\end{enumerate}
\end{corollary}

\begin{remark}\label{R2.6-}
\begin{enumerate}
\item
In Theorem~\ref{T2.4}(i)(a), Corollary~\ref{C2.3}(i) and Corollary~\ref{C2.4}(ii), it suffices to assume that ${\varphi\in\mathcal{C}}$.
\item
In the H\"older case, i.e. when
$\varphi(t):=\tau\iv t^q$ for some $\tau>0$ and $q>0$ and all $t>0$, we have $\varphi'(t)=q\tau\iv t^{q-1}$ and $\frac{\varphi(t)}{t}=\tau\iv t^{q-1}$.
Thus,
the implication (iii) \folgt (iv) in Corollary~\ref{C2.4} is trivially satisfied when $q\le1$, while the opposite implication holds when $q\ge1$, i.e. $\varphi$ is convex.
Moreover, if $q\le1$, then condition \eqref{C2.6-2} is satisfied (as equality) with $\ga=q\iv$.

When
$\varphi(t):=t^q$ for some $q>0$ and all $t>0$, definition \eqref{C2.6-3} reduces to \cite[(3.13)]{KruLopYanZhu19}.
In particular, if $q=1$, it coincides with the strict outer subdifferential slope of $f$ at $\bx$ given by the last expression in \eqref{R2.3-1}.
\item
Condition \eqref{C2.6-2} can be replaced by the following weaker condition: ${\limsup_{t\downarrow0}\frac{\varphi(t)}{t\varphi'(t)}<{+\infty}}$.
\end{enumerate}
\end{remark}

It can be convenient to reformulate Theorem~\ref{T2.4}
using the function $\psi:=\varphi\iv\in\mathcal{C}^1$ instead of $\varphi$.

\begin{corollary}\label{C2.6-}
Suppose $X$ is a metric space, $f:X \rightarrow \mathbb{R}_\infty$, $\bx\in X$, ${\psi\in\mathcal{C}^1}$, $\de\in]0,+\infty]$ and $\mu\in]0,+\infty]$.
\begin{enumerate}
\item
Let $X$ be complete, $f$ be lower semicontinuous, $\al\in]0,1]$, and either $\bx\in[f\le0]$ or ${\de=+\infty}$.
The error bound inequality \eqref{D1-2}
holds for all
$x\in B_{\frac{\de}{1+\al}}(\bx)\cap[0<f<\mu]$,
provided that one of the following conditions is satisfied:
\begin{enumerate}
\item
$\al|\nabla(\psi\iv\circ f)|^\diamond(u)\ge1$
for all $u\in B_\delta(\bx)\cap[0<f<\mu]$ satisfying
\begin{align}\label{C2.6--1}
f(u)<\psi((\max\{\al,1-\al\})\iv d(u,[f\le0]));
\end{align}
\item
$|\breve\nabla f|\in|\mathfrak{D}f|^\circ$ and $\al|\breve\nabla f|(u)\ge\psi'(\psi\iv(f(u)))$
for all $u\in B_\delta(\bx)\cap[0<f<\mu]$ satisfying condition~\eqref{C2.6--1}.
\end{enumerate}
If $\varphi'$ is nonincreasing, then $|\mathfrak{D}f|^\circ$ in {\rm (b)} can be replaced with $|\mathfrak{D}f|$.

\item
If the error bound inequality \eqref{D1-2}
holds for all
$u\in B_{\de}(\bx)\cap[0<f<\mu]$, then
$|{\nabla(\psi\iv\circ f)}|^\diamond(u)\ge1$ for all $u\in B_\de(\bx)\cap[0<f<\mu]$.
\item
Let $X$ be a normed space, and $f$ be convex.
If the error bound inequality \eqref{D1-2}
holds for all
$u\in B_{\de}(\bx)\cap[0<f<\mu]$, then
$|\partial f|(u)\ge\frac{f(u)}{\psi\iv(f(u))}$ for all $u\in B_\de(\bx)\cap[0<f<\mu]$.
Moreover, if $\psi'$ is nonincreasing, particularly if $\psi$ is concave, then
$|\partial f|(u)\ge\psi'(\psi\iv(f(u)))$ for all $u\in B_\de(\bx)\cap[0<f<\mu]$.
\end{enumerate}
\end{corollary}

\begin{remark}
With $\bx\in[f\le0]$, $\de<+\infty$, $\mu=+\infty$ and $\al=1$, Corollary~\ref{C2.6-}(i) with condition (b) and $|\breve\nabla f|=|\sd^Ff|$ strengthens \cite[Theorem 3.2]{YaoZhe16}, while with $\de=\mu=+\infty$ and $\al=1$ it strengthens \cite[Theorem~3.3]{YaoZhe16}.
\end{remark}

\section{Alternative nonlinear error bound conditions} \label{S5}

In this section, we discuss
an alternative set of sufficient and necessary conditions for nonlinear error bounds which instead of values of the given function $f$ employ the distance to the solution set $[f\le0]$.

\begin{theorem}\label{T2.5}
Suppose $X$ is a metric space, $f:X \rightarrow \mathbb{R}_\infty$, $\bx\in X$, ${\varphi\in\mathcal{C}^1}$, $\de\in]0,+\infty]$ and $\mu\in]0,+\infty]$.
\begin{enumerate}
\item
Let $X$ be complete, $f$ be lower semicontinuous, $\varphi$ be concave, ${\al\in]0,1]}$, and either $\bx\in[f\le0]$ or $\de=+\infty$.
Let $|\breve\nabla f|\in|\mathfrak{D}f|$.
The function $f$ admits a $\varphi-$error bound at $\bx$ with $\delta':=\frac{\de}{1+\al}$ and $\mu$,
provided that
\begin{gather}\label{T2.5-1}
\al\varphi'(\varphi\iv((\max\{\al,1-\al\})\iv d(u,[f\le0])))|\breve\nabla f|(u)\ge1
\end{gather}
for all $u\in B_\delta(\bx)\cap[0<f<\mu]$ satisfying condition~\eqref{T2.4-1}.

\item
Let $X$ be a normed space, and $f$ be convex.
If $f$ admits a $\varphi-$error bound at $\bx$ with $\delta$ and $\mu$, then
$\frac{d(u,[f\le0])}{\varphi\iv(d(u,[f\le0]))}|\partial f|(u)\ge1$ for all $u\in B_\de(\bx)\cap[0<f<\mu]$.
If, moreover, $\varphi$ is convex, then
$\varphi'(\varphi\iv(d(u,[f\le0])))|\partial f|(u)\ge1$ for all $u\in B_\de(\bx)\cap[0<f<\mu]$.
\end{enumerate}
\end{theorem}

\begin{proof}
\begin{enumerate}
\item
The assertion is a consequence of Theorem~\ref{T2.4}(i).
It suffices to notice that $\varphi'$ is nonincreasing; hence $\varphi'(\varphi\iv((\max\{\al,1-\al\})\iv d(u,[f\le0])))\le\varphi'(f(u))$ for all $u\in X$ satisfying condition~\eqref{T2.4-1}.

\item
The first part follows from Proposition~\ref{T2.1}(iii) applied with $\tau:=\frac{\varphi\iv(d(x,[f\le0]))}{d(x,[f\le0])}$.
If $\varphi$ is convex, then $\varphi\iv$ is concave, and consequently, $\tau\ge(\varphi\iv)'(d(x,{[f\le0]})) =1/\varphi'(\varphi\iv(d(x,[f\le0])))$.
This proves the second part.
\sloppy
\end{enumerate}
\end{proof}

\begin{remark}\label{R2.6}
\begin{enumerate}
\item
In view of the definition of $|\mathfrak{D}f|$,
Theorem~\ref{T2.5}(i) combines four separate primal and dual sufficient error bound conditions corresponding to $|\breve\nabla f|$ equal to $|\nabla f|^\diamond$, $|\nabla f|$, $|\sd^Cf|$ or $|\sd^Ff|$ (in appropriate spaces).
\item
Compared to Theorem~\ref{T2.4}(i), the statement of
Theorem~\ref{T2.5}(i) contains an additional assumption that $\varphi$ is concave.
This assumption is satisfied, e.g., in the H\"older setting, i.e. when $\varphi(t):=\tau\iv t^q$ for some $\tau>0$ and $q\in]0,1]$, and all $t\ge0$.
In the linear case, i.e. when $q=1$, the sufficient conditions in Theorems~\ref{T2.4} and \ref{T2.5} are equivalent, and reduce to the corresponding ones in Theorem~\ref{T2.2}.

\item
Under the assumption that $\varphi$ is concave and with $\al=1$, the local sufficient conditions in Theorem~\ref{T2.5}(i)
are in a sense weaker than the corresponding conventional ones in Theorem~\ref{T2.4}(i).
Indeed, if $\al=1$ and, given some $\de\in]0,+\infty]$ and $\mu\in]0,+\infty]$, condition (b) in Theorem~\ref{T2.4}(i) is satisfied with some $|\breve\nabla f|\in|\mathfrak{D}f|^\circ$,
then, by Theorem~\ref{T2.4}(i), $\varphi\iv(d(u,[f\le0]))\le f(u)$ for all $u\in B_{\de'}(\bx)$, where $\de':=\de/2$.
(As a consequence, inequality \eqref{T2.4-1} is violated for all $u\in B_{\de'}(\bx)$.)
Then, thanks to the monotonicity of $\varphi'$, for all $u\in B_{\de'}(\bx)$, we have $\varphi'(\varphi\iv(d(u,[f\le0])))\ge\varphi'(f(u))$, and consequently, inequality  \eqref{T2.5-1} is satisfied with the same $|\breve\nabla f|$.

\item
In view of the monotonicity of $\varphi\iv$ and $\varphi'$, one can replace $\max\{\al,1-\al\}$ in \eqref{T2.4-1} and \eqref{T2.5-1} in Theorem~\ref{T2.5}(i) with any positive $\be\le\max\{\al,1-\al\}$.
In particular, one can take $\be:=\al$ or $\be:=1-\al$ (if $\al<1$).
The resulting sufficient conditions are obviously
stronger (hence, less efficient) than those in the current statement.

\if{
\item
Theorem~\ref{T2.5}(i) has been established above as a consequence of Theorem~\ref{T2.4}(i).
It can also be deduced from Theorem~\ref{T2.2}(i) or Proposition~\ref{T2.1}(i).
It suffices to recall that the general nonlinear error bound inequality \eqref{D1-1} can be reduced to the conventional form \eqref{D0-1} by setting $\tau:=\frac{\varphi\iv(d(x,[f\le0]))}{d(x,[f\le0])}$
(This argument was used in the proof of Theorem~\ref{T2.5}(ii).) and make use of the monotonicity of $\varphi'$.
Such an alternative proof is a little longer, but allows to establish additionally the sufficiency of condition of (A$^\prime$).
}\fi

\item
With $\bx\in[f\le0]$, $\mu=+\infty$, $\al<1$, and $\varphi(t):=(\al^q(1-\al)^{1-q}\tau)\iv t^q$ for some $\tau>0$ and $q>0$ and all $t>0$ (H\"older case), Theorem~\ref{T2.5}(i) with $|\breve\nabla f|=|\sd^Ff|$ and $|\breve\nabla f|=|\sd^Cf|$ improves and strengthens \cite[Theorem~3.11]{KruLopYanZhu19}.
\end{enumerate}
\end{remark}

The
local
$\varphi-$error bound sufficient conditions arising from Theorem~\ref{T2.5}(i) are collected in the next three corollaries.

\begin{corollary}\label{C2.5}
Suppose $X$ is a complete metric space, $f:X \rightarrow \mathbb{R}_\infty$ is \lsc, ${\varphi\in\mathcal{C}^1}$, and $\varphi$ is concave.
Let $|\breve\nabla f|\in|\mathfrak{D}f|$.
The function $f$ admits a local $\varphi-$er\-ror bound at $\bx\in[f\le0]$, provided that
$$\varphi'(\varphi\iv(d(x,[f\le0])))|\breve\nabla f|(x)\ge1$$
for all $x\in[f>0]$ near $\bx$ with $f(x)$ near $0$.
\end{corollary}

\begin{corollary}\label{C2.9}
Suppose $X$ is a Banach space, $f:X \rightarrow \mathbb{R}_\infty$ is convex \lsc, $\bx\in[f\le0]$, ${\varphi\in\mathcal{C}^1}$ and $\varphi$ is concave.
Consider the following conditions:
\begin{enumerate}
\item
$f$ admits a local $\varphi-$error bound at $\bx$;


\item
$\varphi'(\varphi\iv(d(x,[f\le0])))|\partial f|(x)\ge1$ for all $x\in[f>0]$ near $\bx$ with $f(x)$ near $0$;

\item
$\frac{d(x,[f\le0])}{\varphi\iv(d(x,[f\le0]))}|\partial f|(x)\ge1$ for all $x\in[f>0]$ near $\bx$ with $f(x)$ near $0$.
\end{enumerate}
Then {\rm (ii) \folgt (i) \folgt (iii)}.
If $\varphi$ is linear, then all the conditions are equivalent.
\end{corollary}

\begin{corollary}\label{C2.10-}
Suppose $X$ is a complete metric space, $f:X \rightarrow \mathbb{R}_\infty$ is \lsc, ${\varphi\in\mathcal{C}^1}$, $\varphi$ is concave, and $\bx\in[f\le0]$.
\begin{enumerate}
\item
The following estimate holds true:
\begin{gather*}
\Er_\varphi f(\bx)\ge
\liminf_{x\to\bx,\,f(x)\downarrow0} \varphi'(\varphi\iv(d(x,[f\le0])))|\nabla f|(x).
\end{gather*}
If $X$ is Banach (Asplund), then $|\nabla f|$ in the first inequality can be replaced with $|\partial^Cf|$ ($|\partial^Ff|$).
\item
If $X$ is Banach and $f$ is convex, then
\begin{gather*}
\Er_\varphi f(\bx)\le
\liminf_{x\to\bx,\,f(x)\downarrow0} \frac{d(x,[f\le0])}{\varphi\iv(d(x,[f\le0]))}|\partial f|(x).
\end{gather*}
Moreover, if $\varphi$ satisfies condition \eqref{C2.6-2},
then $\widehat{|\sd f|}{}_\varphi^>(\bx)\le\Er_\varphi f(\bx)
\le\ga\widehat{|\sd f|}{}_\varphi^>(\bx)$, where
\begin{align*}
\widehat{|\sd f|}{}_\varphi^>(\bx) :=\liminf_{x\to\bx,\,f(x)\downarrow0} \varphi'(\varphi\iv(d(x,[f\le0])))|\sd f|(x);
\end{align*}
as a consequence, $f$ admits a local $(\al\varphi)-$error bound at $\bx$ with some $\al$ satisfying $\widehat{|\sd f|}{}_\varphi^>(\bx)\le\al\iv
\le\ga\widehat{|\sd f|}{}_\varphi^>(\bx)$ if and only if $\widehat{|\sd f|}{}_\varphi^>(\bx)>0$.
\end{enumerate}
\end{corollary}

\begin{remark}
In view of Remark~\ref{R2.6}(iii), when $\varphi$ is concave, each of the sufficient conditions in Corollaries~\ref{C2.5} and \ref{C2.9} is implied by the corresponding condition in Corollaries~\ref{C2.3} and \ref{C2.4}, respectively.
Similarly, condition $\overline{|\sd f|}{}_\varphi^>(\bx)>0$ implies $\widehat{|\sd f|}{}_\varphi^>(\bx)>0$.
\end{remark}

In view of Remark~\ref{R2.6-}(ii),
the next statement is a consequence of Corollaries~\ref{C2.6--} and \ref{C2.10-}.
It recaptures \cite[Corollary~3.35]{KruLopYanZhu19}.

\begin{corollary}
Suppose $X$ is a Banach space, $f:X \rightarrow \mathbb{R}_\infty$ is convex \lsc, $\bx\in[f\le0]$, and $q\in]0,1]$.
The following assertions are equivalent:
\begin{enumerate}
\item
the function
$f$ admits a local error bound of order $q$ at $\bx$, i.e. there exist $\tau>0$ and $\de>0$ such that
\begin{align*}
\tau d(x,[f \le 0])\le(f(x))^q\qdtx{for all}
x\in B_\de(\bx)\cap[f>0],
\end{align*}
\item
$\liminf\limits_{x\to\bx,\,f(x)>0}(f(x))^{q-1}|\partial f|(x)>0.$
\item
$\liminf\limits_{x\to\bx,\,f(x)>0}(d(x,[f\le0]))^{1-\frac{1}{q}} |\sd f|(x)>0.$
\end{enumerate}
\end{corollary}

As mentioned in Remark~\ref{R1.1}(iv),
it is not uncommon to consider nonlinear error bounds with inequality \eqref{D1-2}
where $\psi:=\varphi\iv\in\mathcal{C}$.
Obviously $\varphi$ is concave if and only if $\psi$ is convex.
Next, we reformulate Theorem~\ref{T2.5}
using the function $\psi$ instead of $\varphi$.

\begin{corollary}\label{C2.6}
Suppose $X$ is a metric space, $f:X \rightarrow \mathbb{R}_\infty$, $\bx\in X$, ${\psi\in\mathcal{C}^1}$, $\de\in]0,+\infty]$ and $\mu\in]0,+\infty]$.
\begin{enumerate}
\item
Let $X$ be complete, $f$ be lower semicontinuous, $\psi$ be convex, ${\al\in]0,1]}$, ${\be:=\max\{\al,1-\al\}}$, and either $\bx\in[f\le0]$ or $\de=+\infty$.
Let $|\breve\nabla f|\in|\mathfrak{D}f|$.
The error bound inequality \eqref{D1-2}
holds for all
$x\in B_{\frac{\de}{1+\al}}(\bx)\cap[0<f<\mu]$, provided that
\begin{gather}\label{C2.10-1}
\al|\breve\nabla f|(u)\ge\psi'(\be\iv d(u,[f\le0]))
\end{gather}
for all $u\in B_\delta(\bx)\cap[0<f<\mu]$ satisfying $f(u)<\psi(\be\iv d(u,[f\le0]))$.

\item
Let $X$ be a normed space, and $f$ be convex.
If the error bound inequality \eqref{D1-2}
holds for all
$x\in B_{\de}(\bx)\cap[0<f<\mu]$, then
$|\partial f|(u)\ge\frac{\psi(d(u,[f\le0]))}{d(u,[f\le0])}$ for all $u\in B_\de(\bx)\cap[{0<f<\mu}]$.
Moreover, if $\psi$ is concave, then
$|\partial f|(u)\ge\psi'(d(u,[f\le0]))$ for all $u\in B_\de(\bx)\cap[0<f<\mu]$.
\end{enumerate}
\end{corollary}

\begin{remark}
\begin{enumerate}
\item
Condition \eqref{C2.10-1}
can itself be interpreted as an error bound estimate.
Unlike the conventional upper estimates \eqref{D1-1} and \eqref{D1-2} for the distance to the set $[f\le0]$ in terms of the values of the function $f$, the inequality provides an estimate for this distance in terms of the appropriate slopes of the function $f$.

\item
Remark~\ref{R2.6}(iv) is applicable to Corollary~\ref{C2.6}(i).

\item
With $\bx\in[f\le0]$, $\de<+\infty$, $\mu=+\infty$, and $\al<1$, Corollary~\ref{C2.6}(i) with appropriate slopes recaptures \cite[Theorem~3.1 and Proposition~3.3]{ZhaZhe19} and \cite[Theorems~3.1 and 3.4]{YaoZhe16}.
It is also worth observing that the rather complicated statements in \cite[Theorem~3.1 and Proposition~3.3]{ZhaZhe19} become simpler if the
two parameters $\be>0$ and $\tau>0$ involved in them are replaced with a single parameter $\al:=\be/(\tau+\be)\in]0,1[$ and one employs the function $t\mapsto\psi(\tau\al t)$ instead of $\psi$.

\if{
\item
In view of Remark~\ref{R2.6}(v), condition {\rm (A)} in Corollary~\ref{C2.6}(i) can be replaced with condition
{\rm (A$^\dag$)}.
With such a replacement, it
strengthens \cite[Theorem~3.1]{ZhaZhe19}.
}\fi
\end{enumerate}
\end{remark}

A special case of Corollary~\ref{C2.6}, which can be of interest, is when the function $\psi$ is defined via another function $\nu:\R_+\to\R_+$ as follows: $\psi(t):=\int_0^t\nu(s)ds$ $(t\ge0)$.

\begin{corollary}\label{C2.10}
Suppose $X$ is a complete metric space, $f:X \rightarrow \mathbb{R}_\infty$ is \lsc, $\bx\in X$, $\nu:\R_+\to\R_+$ is nondecreasing, $\nu(t)>0$ for all $t>0$,
\if{
\todo{Is it needed?}
\NDC{1.9.20 I think this assumption is used to ensure $\varphi(+\infty)=+\infty$ so that $\varphi$ is invertible which is needed in the proof of Theorem~\ref{T2.5}.}
}\fi
$\int_0^{+\infty}\nu(t)dt=+\infty$, ${\de\in]0,+\infty]}$, $\mu\in]0,+\infty]$, $\al\in]0,1]$, $\be:=\max\{\al,1-\al\}$, and either $\bx\in[f\le0]$ or $\de=+\infty$.
Let ${|\breve\nabla f|\in|\mathfrak{D}f|}$.
The error bound inequality
\sloppy
\begin{align*}
\int_0^{d(x,[f\le0])}\nu(t)dt\le f(x)
\end{align*}
holds for all
$x\in B_{\frac{\de}{1+\al}}(\bx)\cap[0<f<\mu]$,
provided that
\begin{gather*}
\al|\breve\nabla f|(u)\ge\nu(\be\iv d(u,[f\le0]))
\end{gather*}
for all $u\in B_\delta(\bx)\cap[0<f<\mu]$ satisfying $f(u)<\int_0^{\be\iv d(u,[f\le0])}\nu(s)ds$.
\end{corollary}

\begin{remark}
With $\de=+\infty$ and $\al=1$, Corollary~\ref{C2.10} with $|\breve\nabla f|=|\nabla f|$ recaptures \cite[Theorem~4.3]{CorMot08} and \cite[Theorem~7.1]{AzeCor17}.
With $\de<+\infty$, $\mu=+\infty$ and $\al=1$, it recaptures \cite[Corollary~4.1]{AzeCor14}.
\end{remark}

\if{
\AK{21/07/20.
I need $\psi$ to be nondecreasing to ensure that $\varphi'$ is nonincreasing.
There is no such assumption in \cite[Theorem~7.1]{AzeCor17}, and they claim this to be important \cite[Remark~7.1]{AzeCor17}.
Can it be dropped?}

\NDC{26.7.20
Let me have a closer look at the proofs to see whether it is possible to drop this condition.}

\NDC{26.7.20
I think we do not have the alternative nonlinear sufficient conditions for perturbed error bounds since the perturbed error bound is characterized by the conventional slope characterization ($\varphi'(f(x))|\nabla f|(x)\ge 1$) which is stronger than the alternative one.}
}\fi
\if{
\begin{theorem}
Suppose $X$ is a complete metric space, $f:X \rightarrow \mathbb{R}_\infty$ is \lsc, $\bx\in[f\le0]$, ${\varphi\in\mathcal{C}^1}$, $\varphi'$ is nonincreasing, $\de\in]0,+\infty]$ and $\mu\in]0,+\infty]$.
The error bound estimate \eqref{P2.4-1}
holds for $f$ at $\bx$, provided that one of the following conditions is satisfied:
\begin{enumerate}
\item
$\varphi'(\varphi\iv(d(x,[f\le0])))|\nabla f|(x)\ge1$
for all $x\in X$ satisfying \eqref{T2.4-2};
\item
$X$ is Banach and $\varphi'(\varphi\iv(d(x,[f\le0])))d(0,\partial^Cf(x))\ge1$
for all $x\in X$ satisfying \eqref{T2.4-2};
\item
$X$ is Asplund and $\varphi'(\varphi\iv(d(x,[f\le0])))d(0,\partial^Ff(x))\ge1$
for all $x\in X$ satisfying \eqref{T2.4-2}.
\end{enumerate}
Moreover, {\rm (iii) \folgt (i), (ii) \folgt (i)} and, if $X$ is Asplund, then {\rm (ii) \folgt (iii)}.
\end{theorem}
}\fi


\section{Subregularity of set-valued mappings} \label{S3.1}

In this section, we illustrate the sufficient and necessary conditions for nonlinear error bounds by applying them to characterizing the nonlinear version of the ubiquitous property of \emph{subregularity} (cf. \cite{DonRoc14,Iof17})  of \SVM s.
Nonlinear subregularity has many important applications and has been a subject of intense study in recent years; cf.
\cite{LiMor12,MorOuy15,Kru15.2,Kru16,ZheZhu16}.

Below, $F: X\rightrightarrows Y$ is a set-valued mapping between metric spaces, and $(\bx,\by)\in\gph F$.
Recall that $X\times Y$ is assumed to be equipped with the
maximum distance.

\begin{definition}\label{D3.1}
Let $\varphi\in\mathcal{C}$.
The mapping $F$ is
\begin{enumerate}
\item
$\varphi-$subregular at $(\bx,\by)$ if there exist  $\delta\in]0,+\infty]$ and $\mu\in]0,+\infty]$ such that \begin{align*}
d(x,F\iv(\by))\le\varphi(d(\by,F(x)))
\end{align*}
for all $x\in B_\de(\bx)$  with $d(\by,F(x))<\mu$;
\item
graph $\varphi-$subregular at $(\bx,\by)$ if there exist $\delta\in]0,+\infty]$ and $\mu\in]0,+\infty]$  such that
\begin{align}\label{D3.2-1}
d(x,F\iv(\by))\le\varphi(d((x,\by),\gph F))
\end{align}
for all  $x\in B_\de(\bx)$ with $d((x,\by),\gph F)<\mu$.
\end{enumerate}
\end{definition}

\begin{remark}
Instead of the standard maximum distance employed in \eqref{D3.2-1}, it is common when studying regularity properties of mappings to consider parametric distances of the type $d_\rho((x,y,(u,v)):=\max\{d(x,u),\rho d(y,v)\}$, where $\rho>0$; cf. \cite{AzeCor04,Aze06,Kru15,Kru15.2,Kru16,Iof17}.
This usually gives an additional degree of freedom and leads to sharper conditions.
We avoid doing it here just for simplicity as our main purpose in this section is to provide some illustrations.
All the conditions below can be easily extended to parametric distances.
\end{remark}

The nonlinear property in part (i) of Definition~\ref{D3.1} is a quite common extension of the conventional subregularity; cf., e.g., \cite{Kru16}, while the property in part (ii) extends the (linear) graph subregularity studied 
in 
\cite{JouThi90}.
In the linear case, the two properties are equivalent.
The next proposition establishes quantitative relations between the properties in the nonlinear setting.

\begin{proposition}\label{P5.1}
Let  $\varphi\in\mathcal{C}$,
$\delta\in]0,+\infty]$ and $\mu\in]0,+\infty]$.
\begin{enumerate}
\item
If $F$ is graph $\varphi-$subregular at $(\bx,\by)$ with $\de$ and $\mu$, then it is $\varphi-$subregular at $(\bx,\by)$ with $\de$ and $\mu$.
\item
Suppose $c\varphi(t)\ge t$ for some $c>0$ and all $t\in]0,3\de/2[$.
If $F$ is $\varphi-$subregular at $(\bx,\by)$ with $\de$ and $\mu$, then it is graph $(c+1)\varphi-$subregular at $(\bx,\by)$ with  $\de':=\min\{\de,\mu\}/2$ and any $\mu'\in]0,+\infty]$.
\end{enumerate}
\end{proposition}

\begin{proof}
\begin{enumerate}
\item
Suppose $F$ is graph $\varphi-$subregular at $(\bx,\by)$ with $\de$ and $\mu$.
Let $x\in B_{\delta}(\bx)$ with $d(\by,F(x))<\mu$.
Then
$d((x,\by),\gph F)\le d(\by,F(x))$.
Thus, $d((x,\by),\gph F)<\mu$ and, in view of Definition~\ref{D3.1}(ii),
$d(x,F\iv(\by))\le\varphi(d((x,\by),\gph F))\le \varphi(d(\by,F(x)))$.
Hence, $F$ is $\varphi-$subregular at $(\bx,\by)$ with $\de$ and $\mu$.

\item
Suppose $F$ is $\varphi-$subregular at $(\bx,\by)$ with $\de$ and $\mu$.
Let $x \in B_{\delta'}(\bx)$, where
${\de':=\min\{\de,\mu\}/2}$.
Observe that $d((x,\by),\gph F)\le d(x,\bx)$ and, if $d((u,v),(x,\by))\le d(x,\bx)$, then
$d((u,v),(\bx,\by))\le 2 d(x,\bx)<2\de'$.
Hence,
\begin{align}\label{P5.1P1}
	d((x,\by),\gph F)=\inf_{(u,v)\in\gph F\cap B_{2\de'}(\bx,\by)}
d((x,\by),(u,v)).
\end{align}
For any $(u,v)\in\gph F\cap B_{2\de'}(\bx,\by)$, we have $u\in B_{\de}(\bx)$, and
$d(\by,F(u))\le d(\by,v)<2\de'=\min\{\de,\mu\}$.
Thus, by Definition~\ref{D3.1}(i) and the monotonicity of $\varphi$,
\sloppy
\begin{align*}
	d(u,F\iv(\by))\le\varphi(d(\by,F(u)))\le\varphi(d(\by,v)).
\end{align*}
Moreover, $d(x,u)<3\de'\le 3\de/2$,
and consequently, in view of the assumption on $\varphi$,
\begin{align*}
d(x,F\iv(\by))
&\le d(x,u)+d(u,F\iv(\by))
\le d(x,u)+\varphi(d(\by,v))\\
&\le c\varphi(d(x,u))+\varphi(d(\bar y,v))
\le (c+1)\varphi(d((x,\by),(u,v))).
\end{align*}
Taking infimum over $(u,v)\in\gph F\cap B_{2\de'}(\bx,\by)$ and using \eqref{P5.1P1}, we arrive at
$d(x,F\iv(\by))\le (c+1)\varphi(d((x,\by),\gph F))$.
Hence, $F$ is graph $(c+1)\varphi-$subregular at $(\bx,\by)$ with $\delta'$ and any ${\mu'\in]0,+\infty]}$.
\end{enumerate}
\end{proof}

Observe that graph subregularity of $F$ in Definition~\ref{D3.1}(ii) is precisely the error bound property
of the Lipschitz continuous function $x\mapsto d((x,\by),\gph F)$.
When applied to this function, formulas \eqref{slopes} and \eqref{sld}
lead to the following definitions of
the primal and dual slopes of the \SVM\ $F$ at $\bx$:
\begin{gather*}
|\nabla F|^\diamond(x):=\sup\limits_{u\ne x}
\dfrac{ [ d((x,\by), \gph F)- d((u,\by), \gph F)]_+}{d(u,x)},\\
|\nabla F|(x):=\limsup_{u\rightarrow x,\,u\ne x}\dfrac{ [ d((x,\by), \gph F)- d((u,\by), \gph F)]_+}{d(u,x)},
\\
|\sd^C F|(x):=d(0,\partial^Cd((\cdot,\by),\gph F)(x)),\quad
|\sd^F F|(x):=d(0,\partial^Fd((\cdot,\by),\gph F)(x)).
\end{gather*}
Note that they differ from the corresponding slopes used in \cite{Kru15,Kru15.2,Kru16}.
We use below the collections of
slope operators of $F$ (being realizations of the corresponding collections of $f$) defined recursively as follows:
\begin{enumerate}
\item
$|\mathfrak{D}F|^\circ:=\{|\nabla F|\}$;
\item
if $X$ is Banach, then
$|\mathfrak{D}F|^\circ :=|\mathfrak{D}F|^\circ\cup\{|\sd^C F|\}$;
\item
if $X$ is Asplund, then
$|\mathfrak{D}F|^\circ:=|\mathfrak{D}F|^\circ\cup\{|\sd^F F|\}$;
\item
${|\mathfrak{D}F|}:=
|\mathfrak{D}F|^\circ\cup\{|\nabla F|^\diamond\}$;
\item
${|\mathfrak{D}F|^\dag}:=
|\mathfrak{D}F|\setminus\{|\sd^F F|\}$.
\end{enumerate}

The next two statements are consequences of Theorems~\ref{T2.4} and \ref{T2.5}, respectively.
Their first parts extend \cite[Theorem~2.4]{Jou00} and \cite[Theorem~2.53]{Iof17}.

\begin{proposition}\label{P3.4}
Suppose $X$ and $Y$ are metric spaces, ${\varphi\in\mathcal{C}^1}$, $\de\in]0,+\infty]$ and $\mu\in]0,+\infty]$.
\begin{enumerate}
\item
Let $X$ be complete and $\al\in]0,1]$.
The mapping $F$ is graph $\varphi-$subregular at $(\bx,\by)$ with $\delta':=\frac{\de}{1+\al}$ and $\mu$,
provided that one of the following conditions is satisfied:
\begin{enumerate}
\item
$\al|\nabla(\varphi\circ d((\cdot,\by),\gph F))|^\diamond(u)\ge1$
for all $u\in B_\delta(\bx)\setminus F\iv(\by)$ with $d((u,\by),\gph F)<\mu$ satisfying
\begin{align}\label{C3.3-1}
\max\{\al,1-\al\}\varphi(d((u,\by),\gph F))<d(u,F\iv(y));
\end{align}
\item
$|\breve\nabla F|\in|\mathfrak{D}F|^\circ$ and $\al\varphi'(d((u,\by),\gph F))|\breve\nabla F|(u)\ge1$
for all $u\in B_\delta(\bx)\setminus F\iv(\by)$ with $d((u,\by),\gph F)<\mu$ satisfying condition~\eqref{C3.3-1}.
\end{enumerate}
\sloppy
If $\varphi'$ is nonincreasing, then $|\mathfrak{D}F|^\circ$ in {\rm (b)} can be replaced with $|\mathfrak{D}F|$.
\item
If $F$ is $\varphi-$graph regular at $(\bx,\by)$ with $\delta$ and $\mu$, then
$|{\nabla(\varphi\circ d((\cdot,\by),\gph F))}|^\diamond(u)\ge1$ for all $u\in B_\de(\bx)\setminus F\iv(\by)$ with $d((u,\by),\gph F)<\mu$.
\item
Let $X$ and $Y$ be normed spaces, and $\gph F$ be convex.
If $F$ is graph $\varphi-$subregular at $(\bx,\by)$ with $\delta$ and $\mu$, then
$\frac{\varphi(d((u,\by),\gph F))}{d((u,\by),\gph F)}|\partial F|(u)\ge1$ for all $u\in B_\de(\bx)\setminus F\iv(\by)$ with $d((u,\by),\gph F)<\mu$.\\
If, moreover, $\varphi'$ is nondecreasing, particularly if $\varphi$ is convex, then
$\varphi'(d((u,\by),\gph F))|\partial F|(u)\ge1$ for all $u\in B_\de(\bx)\setminus F\iv(\by)$ with $d((u,\by),\gph F)<\mu$.
\end{enumerate}
\end{proposition}

\begin{proposition}\label{P3.5}
Suppose $X$ and $Y$ are metric spaces, ${\varphi\in\mathcal{C}^1}$, $\de\in]0,+\infty]$ and $\mu\in]0,+\infty]$.
\begin{enumerate}
\item
Let $X$ be complete, $\varphi'$ be nonincreasing, and ${\al\in]0,1]}$.
Let $|\breve\nabla F|\in|\mathfrak{D}F|$.
The mapping $F$ is graph $\varphi-$subregular at $(\bx,\by)$ with $\delta':=\frac{\de}{1+\al}$ and $\mu$,
provided that
\begin{gather*}
\al\varphi'(\varphi\iv((\max\{\al,1-\al\})\iv d(u, F\iv(\by))))|\breve\nabla F|(u)\ge1
\end{gather*}
for all $u\in B_\delta(\bx)\setminus F\iv(\by)$ with $d((u,\by),\gph F)<\mu$ satisfying condition~\eqref{C3.3-1}.
\item
Let $X$ and $Y$ be normed spaces, and $\gph F$ be convex.
If $F$ is graph $\varphi-$subregular at $(\bx,\by)$ with $\delta$ and $\mu$, then
$\frac{d(u,F\iv(\by))}{\varphi\iv(d(u,F\iv(\by)))}|\partial F|(u)\ge1$ for all $u\in B_\de(\bx)\setminus F\iv(\by)$ with $d((u,\by),\gph F)<\mu$.\\
If, moreover, $\varphi'$ is nondecreasing, particularly if $\varphi$ is convex, then\\
$\varphi'(\varphi\iv(d(u,F\iv(\by))))|\partial F|(u)\ge1$ for all $u\in B_\de(\bx)\setminus F\iv(\by)$ with $d((u,\by),\gph F)<\mu$.
\end{enumerate}
\end{proposition}

\if{
\red{
\begin{remark}
The sufficient conditions in Propositions~\ref{P3.4} and \ref{P3.5} improve \cite[Theorem~2.53]{Iof17} and \cite[Theorem~2.4]{Jou00} and the linear setting.
\end{remark}
}

\AK{22/11/20.
The subsection is not finished.

1) $f$ should be excluded from the statements, i.e. the slopes need to be reformulated in terms of $F$.

2) Can the resulting `slopes' of $F$ be compared with those used in \cite{Kru15,Kru15.2,Kru16}?

3) Should a parametric norm be used on $X\times Y$?

4) A comparison with the existing results is missing.
}

\NDC{23.11.20
2) In \cite{Kru15,Kru15.2,Kru16}, the subregularity of $F$ is studied by using the slope of the function $(x,y)\mapsto d(x,\by) +i_{\gph F}(x,y)$.
In my understanding, using this function (instead of the original one, i.e. $x\mapsto d(\by,F(x))$) has two advantages.
First, this function is lower semicontinuous when $\gph F$ is closed. Second, it allows us to compute the subdifferentials explicitly thanks to the subdifferential presentations of the distance function and the indicator function.
Of course, we can study the graph subregularity by employing the function $(x,u,v)\mapsto d((x,\by),(u,v))+i_{\gph F}(u,v)$.
However, I think our aim here is to take advantage of the continuity of the function $x\mapsto d((x,\by),\gph F)$.
As a consequence, I think the slopes (of the function of one variable) here cannot be compared with the ones  (of the function of two variables) in \cite{Kru15,Kru15.2,Kru16}.

3) Historically, Jourani and Thibault \cite{JouThi90} used the usual distance in the product space when defining the property;
Ioffe \cite[Definition~2.60]{Iof17} defined the graph regularity by using the parametric distance, but he did not introduce the modulus (constant $\al$) in the definition of the property.
I think there is no issue with putting a parameter to the distance function in Definition~\ref{D3.1}(ii).
In the initial version, I put a parameter in, but then I take it out since I think we will not use this parameter.}
}\fi

\section{Convex semi-infinite optimization} \label{S6}

In this section, we
consider a canonically perturbed convex semi-infinite optimization problem:
\begin{equation*}
\begin{aligned}
P(c,b):\quad& \text{minimize}
&&\psi(x)+\langle c,
x\rangle\\
&\text{subject to} &&g_t(x)\leq b_t,\;t\in T,
\end{aligned}
\end{equation*}
where $x\in\R^n$ is the vector of variables, $c\in\R^n$, $\langle \cdot,\cdot\rangle$ represents the usual inner product in $\R^n$, $T$ is a proper compact subset of a metric space $Z$,
$\psi:\R^{n}\rightarrow\R$
and $g_t:\R^n\rightarrow \R$ $(t\in T)$ are convex functions, $(t,x)\mapsto g_t(x)$ is assumed to be lower semicontinuous on $ T\times\R^n$, and $b\in C(T,\R)$ (the space of continuous functions from $T$ to $\R$).
In this setting, the pair
$(c,b)\in \R^{n}\times C(T,\R)$ is regarded as the parameter to be perturbed.
The norm in this parameter space is given by
$\|(c,b)\|:=\max\{\|c\|,\|b\|_{\infty}\},$
where $\|\cdot\|$ is any given norm in $\R^n$  and $\|b\|_{\infty}:=\max_{t\in T}|b_{t}|$.

The \emph{solution},
\emph{feasible set} and \textit{level set} mappings corresponding to the above problem are the set-valued mappings
defined, respectively, by
\begin{gather}\label{S}
\mathcal{S}(c,b):=\{x\in\R^{n}\mid x\;\;\mathrm{solves}\;\;P(c,b)\},\quad
(c,b)\in \R^{n}\times C(T,\R),
\\\notag
\mathcal{F}(b):=\{x\in \R^{n}\mid g_{t}(x)\leq b_{t},\;t\in T\},\quad
b\in C(T,\R),
\\\label{L}
\mathcal{L}(\alpha ,b):=\{x\in \mathcal{F}(b)\mid \psi(x)+\langle \bar{c},x\rangle \leq \alpha\},
\quad (\al,b)\in\R\times C(T,\R).
\end{gather}
If $c:=\bar c$ in \eqref{S} is fixed, then $\mathcal{S}$ reduces to the partial solution mapping $\mathcal{S}_{\bar c}:C(T,\R)\rightrightarrows
\R^{n}$ given by $\mathcal{S}_{\bar c}(b)=\mathcal{S}(\bar c,b)$.

Our goal in this section is to use nonlinear error bound conditions for analyzing nonlinear \emph{calmness} of the mappings $\mathcal{S}$, $\mathcal{S}_{\bar c}$ and $\mathcal{L}$.

\begin{definition}\label{D3.3}
Let $F: Y\rightrightarrows X$ be a set-valued mapping between metric spaces, $(\by,\bx)\in\gph F$, and $ \varphi\in\mathcal{C}$.
The mapping $F$ is $\varphi-$calm
at $(\by,\bx)$ if there exist $\de\in]0,+\infty]$ and
$\mu\in]0,+\infty]$ such that
\begin{align}\label{D4.3-1}
d(x,F(\by))\le\varphi(d(y,\by))
\end{align}
for all $y\in Y$ with $d(y,\by)<\mu$ and $x\in F(y)\cap B_\de(\bx)$.
\end{definition}

Calmness of set-valued mappings plays an important
role in optimization theory, cf. \cite{DonRoc14,Iof17}.
It is easy to check that
$F$ is $\varphi-$calm at $(\by,\bx)\in\gph F$ with some $\de$ and $\mu$ if and only if $F\iv$ is  $\varphi-$subregular at  $(\bx,\by)$ with the same $\de$ and $\mu.$

C\'anovas et al. \cite{CanKruLopParThe14} examined the problem $P(c,b)$ in the particular case when $g_t$ $(t\in T)$ are linear functions and $\psi\equiv0$, and obtained
estimates for the calmness modulus.
In \cite{CanHanLopPar08}, C\'anovas et al. studied the modulus of metric regularity of the solution mapping.
Kruger et al. \cite{KruLopYanZhu19} established characterizations of H\"older calmness of the solution mapping by employing the error bound theory.
Motivated by the latter paper, we
establish characterizations of the calmness in the nonlinear setting.

From now on, we assume a point $((\bar{c},\bar{b}),\bar{x})\in\gph\mathcal{S}$ to be given.
Obviously, if
$\mathcal{S}$ is $\varphi-$calm at $((\bar{c},\bar{b}),\bar{x})$ for some $\varphi\in\mathcal{C}$, then
$\mathcal{S}_{\bar{c}}$ is $\varphi-$calm at $(\bar{b},\bar{x})$.
We are going to employ
the following \lsc\ convex function:
\begin{gather}\label{barf}
{f}(x):=\max\{\psi(x)-\psi(\bar{x})+\langle \bar{c},x-\bar{x}\rangle,\;\sup_{t\in T}(g_{t}(x)-\bar{b}_t)\},\quad x\in\R^n.
\end{gather}
Observe that
\begin{gather}\label{marco6}
\mathcal{S}(\bar{c},\bar{b})=\left[{f}=0\right]=\left[{f}
\leq0\right]  =\mathcal{L}(\psi(\bar{x})+\langle\bar{c},\bar{x}\rangle,\bar{b}),
\\\notag
{f}_{+}(x)
=d((\psi(\bar{x})+\langle\bar{c},\bar{x}\rangle,\bar{b}), \mathcal{L}^{-1}(x))\qdtx{for all} x\in\R^n.
\end{gather}
As a consequence, we have the following statement.

\begin{proposition}\label{P6.9}
Let $\varphi\in\mathcal{C}$.
The mapping $\mathcal{L}$ is $\varphi-$calm at $(\bx,(\psi(\bx)+\langle\bar c,\bx\rangle,\bar b))$ with some $\de\in]0,+\infty]$ and
$\mu\in]0,+\infty]$ if and only if $f$ admits a $\varphi-$error bound at $\bx$ with the same $\de$ and $\mu$.
\end{proposition}

The assertions in the next proposition are extracted from \cite[Proposition~4.5 \& Theorem~4.7]{KruLopYanZhu19} and their proofs.
The set of \textit{active indices} at $x\in \mathcal{F}(b)$ is defined by
$T_{b}(x):=\{{t\in T}\mid g_{t}(x)=b_{t}\}.$
The problem $P(c,b)$ satisfies the \textit{Slater condition} if
there exists an $\hat{x}\in\R^{n}$ such that $g_{t}(\hat{x})<b_{t}$ for all $t\in T$.

\begin{proposition}\label{P4.6}
Let $P(\bar{c},\bar{b})$ satisfy the Slater condition.
\begin{enumerate}
\item
There exist $\de>0$, $\mu>0$ and $M>0$ such that
\begin{align}\label{P4.6-1}
\psi(x)-\psi(\bar{x})+\langle \bar{c},x-\bar{x}\rangle \leq M\|(c,b)-(\bar{c},\bar{b})\|
\end{align}
for all $(c,b)\in B_{\mu}(\bar c,\bar b)$ and $x\in \mathcal{S}(c,b)\cap B_\de(\bx)
$.
\item
If $x^{n}\to\bar{x}$ with ${f}(x^{n})\downarrow0$,
then
\begin{enumerate}
\item
there exists
a sequence $\{b^{n}\}_{n\in\mathbb{N}}\subset C(T,\R)$
such that
$x^{n}\in\mathcal{F}(b^{n})$ and
$\|b^n-\bar b\|_{\infty}\le Nf(x^n)$
for some $N>0$ and all $n\in\N$;

\item
there exist
a finite subset $T_{0}\subset\cap_{n\in\N}T_{b^{n}}(x^{n})$, and $\gamma _{t}>0$, $u_{t}\in \partial g_{t}(\bar{x})$ $(t\in T_{0})$ and $u\in \partial\psi(\bar{x})$ such that
$-(\bar{c}+u)\in\sum_{t\in T_0}\gamma_{t}u_{t}$.
\end{enumerate}
\end{enumerate}
\end{proposition}

The next proposition, {a nonlinear counterpart of \cite[Propositions~4.4]{KruLopYanZhu19}}, gives {a} sufficient condition
for
the calmness of the level set mapping $\mathcal{L}$.

\begin{proposition}\label{P3.2}
Let $P(\bar{c},\bar{b})$ satisfy the Slater condition and  $\varphi\in\mathcal{C}^1$ satisfy the following condition: 
\begin{align}\label{P3.2-0}
\forall N>0\quad\exists \ga>0\qdtx{such that} \frac{\varphi(Nt)}{t}\le\ga\varphi'(t)\qdtx{for all}t>0.
\end{align}
If $\mathcal{L}$ is not $(\al\varphi)-$calm at $((\psi(\bar
{x})+\langle\bar{c},\bar{x}\rangle,\bar{b}),\bar{x})$ for all $\al>0$, then there exist sequences $x^{n}\to\bar{x}$ and $\{b^{n}\}_{n\in\mathbb{N}}\subset C(T,\R)$
such that $x^{n}\in\mathcal{F}(b^{n})$
$(n\in\N)$, ${f}(x^{n})\downarrow0$ and
\begin{align}\label{P3.2-1}
\lim_{n\rightarrow+\infty} \frac{\varphi(\|b^{n}-\bar{b}\|_{\infty})} {d(x^{n},\mathcal{S}_{\bar{c}}(\bar{b}))}=0.
\end{align}
\end{proposition}

\begin{proof}
Suppose $\mathcal{L}$ is not $(\al\varphi)-$calm at $((\psi(\bar
{x})+\langle\bar{c},\bar{x}\rangle,\bar{b}),\bar{x})$ for all $\al>0$.
By Corollary~\ref{C2.6--}(ii), 
there exists a sequence $x^{n}\to\bar{x}$ with ${f}(x^{n})\downarrow0$ such that
\begin{gather*}
\lim\limits_{n\rightarrow+\infty}\varphi'({f}(x^n))|\sd f|(x^n)=0.
\end{gather*}
By Lemma~\ref{P1.1}(iii),
\begin{align*}
|\sd f|(x^n)=|\nabla f|^\diamond(x^n)
\ge\sup_{u\in\mathcal{S}(\bar c,\bar b)}\frac{f(x^n)}{\|x^n-u\|} =\frac{f(x^n)}{d(x^{n},\mathcal{S}(\bar{c},\bar{b}))}, \quad n\in\N.
\end{align*}
By Proposition~\ref{P4.6}(ii),
there exist a sequence $\{b^{n}\}_{n\in\mathbb{N}}\subset C(T,\R)$ and a number ${N\in\N}$
such that $x^{n}\in\mathcal{F}(b^{n})$ and
$\|b^n-\bar b\|_{\infty}\le Nf(x^n)$
for all $n\in\N$.
Then, with $\ga>0$ corresponding to $N$ in view of condition \eqref{P3.2-0}, we have
\begin{align*}
0&\le\lim_{n\rightarrow+\infty}\dfrac{\varphi(\| b^{n}-\bar{b}\|_{\infty})}
{d(x^{n},\mathcal{S}(\bar{c},\bar{b}))}
\le\lim_{n\rightarrow+\infty}|\sd f|(x^n)\frac{\varphi(\|b^{n}-\bar{b}\|_{\infty})}{{f}(x^{n})}\\
&\le\lim_{n\rightarrow+\infty}|\sd f|(x^n)\frac{\varphi(N{f}(x^{n}))}{{f}(x^{n})} \le\ga\lim_{n\rightarrow+\infty}|\sd f|(x^n)\varphi'({f}(x^{n}))=0.
\end{align*}
This completes the proof.
\end{proof}

\begin{remark}
\begin{enumerate}
\item
Condition \eqref{P3.2-0} implies condition \eqref{C2.6-2}.
It is satisfied, e.g., in the H\"older case, i.e. when
$\varphi(t)=\tau\iv t^q$ for some
$\tau>0$ and $q\in]0,1]$, or more generally, when
$\varphi(t)=\tau\iv(t^q+\be t)$ for some $\tau>0$, $\be>0$ and $q\in]0,1]$.
\item
Condition \eqref{P3.2-0} can be replaced by the following weaker condition: for any $N>0$, $\limsup_{t\downarrow0}\frac{\varphi(Nt)}{t\varphi'(t)}<{+\infty}$.
\end{enumerate}
\end{remark}

\begin{proposition}
Let $P(\bar{c},\bar{b})$ satisfy the Slater condition,
and $\varphi\in\mathcal{C}$.
If $f$ admits a $\varphi-$error bound at $\bx$ with some
$\de\in]0,+\infty]$ and
$\mu\in]0,+\infty]$, then there exists an $\al>0$ such that
$\mathcal{S}$ is $\phi_\al-$calm at $((\bar c,\bar b),\bx)$ with some $\de'\in]0,\de[$ and $\mu'\in]0,\mu[$, where
$\phi_\al(t):=\varphi(\al t)$
$(t\ge0)$.
\end{proposition}

\begin{proof}
Let ${f}$ have a $\varphi-$error bound at $\bar{x}$ with some $\delta\in]0,+\infty]$ and $\mu\in]0,+\infty]$.
By Proposition~\ref{P4.6}(i), there exist $\de'\in]0,\de[$, $\mu'\in]0,\mu[$ and $M>0$ such that inequality \eqref{P4.6-1} holds for all $(c,b)\in B_{\mu'}(\bar c,\bar b)$ and
$x\in\mathcal{S}(c,b)\cap B_{\de'}(\bx)
$.
It follows from \eqref{D1-1}, \eqref{marco6} and \eqref{P4.6-1} that,
for all $(c,b)\in B_{\mu'}(\bar c,\bar b)$ and $x\in\mathcal{S}(c,b)\cap B_{\de'}(\bx)\cap[f>0]$ (hence, $x\in\mathcal{F}(b)$),
\sloppy
\begin{align*}\notag
d(x,\mathcal{S}(\bar{c},\bar{b}))=
&d(x,[{f}\leq 0])
\leq {\varphi(f(x))}
\\\notag
\le&
\varphi(\max\{[\psi(x)-\psi(\bar{x})+\langle \bar{c},x-\bar{x}\rangle]_+,
\;\sup_{t\in T}[b_{t}-\bar{b}_{t}]_+\})
\\
\le&\varphi(\max\{M,\mu'\}\|(c-\bar c,b-\bar b)\|)
=\varphi(\al\|(c-\bar c,b-\bar b)\|),
\end{align*}
where $\al:=\max\{M,\mu'\}$.
Hence, $\mathcal{S}$ is $\phi_\al-$calm at $((\bar{c},\bar{b}),\bar{x})$ with
$\de'$ and $\mu'$.
\end{proof}

\begin{proposition}
Let $\psi$ and $g_{t}$ $(t\in T)$ be linear, $P(\bar{c},\bar{b})$ satisfy the Slater condition, and $\varphi\in\mathcal{C}^1$ satisfy condition \eqref{P3.2-0}.
If $\mathcal{S}_{\bar{c}}$ is $\varphi-$calm at $(\bar{b},\bar{x})$, then there exists an $\al>0$ such that
$\mathcal{L}$ is $(\al\varphi)-$calm at $((\psi(\bar
{x})+\langle\bar{c},\bar{x}\rangle,\bar{b}),\bar{x})$.
\end{proposition}

\begin{proof}
Suppose $\mathcal{L}$ is not $(\al\varphi)-$calm at $((\psi(\bar{x})+\langle\bar{c},\bar{x}\rangle,\bar{b}),\bar{x})$
for all $\al>0$.
By Proposition~\ref{P4.6}(ii) and Proposition~\ref{P3.2}, there exist sequences $x^{n}\to\bar{x}$ and $b^{n}\to\bar{b}$ such that
$x^{n}\in\mathcal{F}(b^{n})$ $(n\in\N)$, and conditions \eqref{P3.2-1} and (b) in Proposition~\ref{P4.6}(ii) are satisfied.
By continuity, we can assume that $P(\bar{c},b^{n})$ satisfies the Slater condition for all sufficiently large $n$.
It readily follows from condition (b) in Proposition~\ref{P4.6}(ii) in the linear setting that $x^{n}\in\mathcal{S}_{\bar{c}}(b^{n})$ for all sufficiently large $n$.
Thus, $\mathcal{S}_{\bar{c}}$ is not $\varphi-$calm at $(\bar{b},\bar{x})$.
\end{proof}

\section*{Acknowledgements}

We would like to thank the referees for their constructive comments.
We are also very grateful to the following colleagues, who provided feedback on the preprint of the paper and made many helpful suggestions:
Jean-No\"el Corvellec,
Joydeep Dutta,
Marco L\'opez,
Ngai Huynh Van,
Michel Th\'era and
Jane Ye.

\section*{Disclosure statement}

No potential conflict of interest was reported by the authors.

\section*{Funding}

The research was supported by the Australian Research Council, project DP160100854.
The second author benefited from the support of the European
Union's Horizon 2020 research and innovation programme under the Marie Sk{\l}odowska--Curie Grant
Agreement No. 823731 CONMECH, and Conicyt REDES program 180032.

\addcontentsline{toc}{section}{References}
\bibliography{EB,BUCH-kr,Kruger,KR-tmp}
\bibliographystyle{tfnlm}
\end{document}